\documentclass{article}

\usepackage[utf8]{inputenc}
\usepackage{amsmath,amsfonts,amssymb}
\usepackage{amsthm}
\usepackage{subcaption}
\usepackage{enumitem}

\usepackage[backend=biber,
            backref=true,
						maxnames=99,
            isbn=false,
            doi=false,
            sortcites=false,
            style=authoryear-icomp]{biblatex}
\usepackage{tikz}
\usetikzlibrary{matrix}

\usepackage[unicode,colorlinks=true,linkcolor=blue,urlcolor=magenta, citecolor=blue]{hyperref}
\usepackage[capitalize, nameinlink]{cleveref}

\newtheorem{theoremintro}{Theorem}

\newtheorem*{corollaryintro}{Corollary}
\newtheorem{theorem}{Theorem}[section]
\newtheorem{definition}[theorem]{Definition}

\newtheorem{proposition}[theorem]{Proposition}
\newtheorem{lemma}[theorem]{Lemma}

\newtheorem{step}{Step}

\theoremstyle{remark}

\newtheorem*{acknowledgments}{Acknowledgments}
\newtheorem{remark}[theorem]{Remark}

\crefformat{theoremintro}{Theorem #2#1#3}
\crefrangeformat{theoremintro}{Theorems~#3#1#4 to~#5#2#6}
\crefmultiformat{theoremintro}{Theorems~#2#1#3}%
{ and~#2#1#3}{, #2#1#3}{ and~#2#1#3}

\newlist{stepitems}{enumerate}{1}
\setlist[stepitems]{label=(\thestep.\alph*), ref=(\thestep.\alph*)}
\crefformat{stepitemsi}{Property #2#1#3}
\crefrangeformat{stepitemsi}{Properties~#3#1#4 to~#5#2#6}
\crefmultiformat{stepitemsi}{Properties~#2#1#3}%
{ and~#2#1#3}{, #2#1#3}{ and~#2#1#3}

\newcommand{\co}{\colon\thinspace}

\DeclareMathOperator{\Id}{Id}
\DeclareMathOperator{\id}{Id}
\DeclareMathOperator{\Int}{Int}
\DeclareMathOperator{\Conj}{c}
\newcommand{\conj}[2]{\Conj_{#1}\left(#2\right)}
\DeclareMathOperator{\codim}{codim}

\newcommand{\tildeG}{\scalebox{.9}{$\widetilde G$}}

\newcommand{\D}{\mathcal{D}}
\newcommand{\Dc}{\mathcal{D}_{\!c}}
\newcommand{\Do}{\mathcal{D}_{\!o}}
\newcommand{\tildeDo}[1]{\rlap{\raisebox{-0.2ex}{$\widetilde{\Do}$}}%
  \rlap{\smash{\textcolor{white}{\rule[-1.1ex]{2em}{2.7ex}}}}\Do(#1)}
\DeclareMathOperator{\supp}{supp}
\DeclareMathOperator{\crit}{Crit}

\let\U\relax
\DeclareMathOperator{\U}{U}

\DeclareMathOperator{\Hom}{Hom}

\newcommand{\J}{\mathrm{J}}

\let\P\relax
\DeclareMathOperator{\P}{\mathcal{P}} %

\DeclareMathOperator{\Lie}{\mathcal{L}} %

\newcommand{\A}{\mathbb{A}}

\renewcommand{\k}{\mathbb{k}}
\newcommand{\h}{\mathfrak{h}}
\newcommand{\W}{\mathcal{W}} %
\def\M{\mathcal{M}}

\newcommand\rst[1]{|_{#1}}

\newcommand\V[1]{\mathcal{V}_{\!\!#1}}

\renewcommand{\d}{d}
\newcommand{\intprod}{%
    \mathbin{\scalebox{1.5}{$\lrcorner$}}%
}
\AtBeginDocument{\renewcommand{\setminus}{\smallsetminus}}
\newcommand{\del}{\partial}
\newcommand{\R}{\mathbb R}

\newcommand{\Z}{\mathbb Z}
\newcommand{\N}{\mathbb N}

\renewcommand{\j}{\jmath}

\newcommand{\Rel}{\mathcal{R}} %
\newcommand{\RelLeg}{\mathcal{S}} %

\title{Contactomorphism groups and Legendrian flexibility}

\author{Sylvain \textsc{Courte} \and Patrick \textsc{Massot}}

\AtEndDocument{\bigskip{\footnotesize%
    \noindent
    S.~Courte: \textsc{Université Grenoble Alpes, IF, CNRS, F-38000 Grenoble, France.} \par
    \noindent
    \texttt{sylvain.courte@univ-grenoble-alpes.fr} \par
    \addvspace{\medskipamount}
    \noindent
    P.~Massot: \textsc{Laboratoire de Mathématiques d’Orsay, Univ. Paris-Sud,
    CNRS, Université Paris-Saclay, 91405 Orsay, France.}\par
    \noindent
    \texttt{patrick.massot@u-psud.fr} \par
}}

\begin{document}

\maketitle

\begin{abstract}
  We explain a connection between the algebraic and geometric properties of
  groups of contact transformations, open book decompositions, and flexible
  Legendrian embeddings. The main result is that, if a closed contact manifold
  $(V, \xi)$ has a supporting open book whose pages are flexible Weinstein
  manifolds, then the connected component $G$ of the identity in its automorphism
  group is a uniformly simple group: for every
  non-trivial element $g$, every other element is a product of at most 
  $128(\dim V + 1)$ conjugates of $g^{\pm1}$. In particular any
  conjugation invariant norm on this group is bounded. We also prove
  the later statement still holds for the universal cover of $G$.
\end{abstract}

\section*{Introduction}

In this paper, $(V, \xi)$ will always denote a connected manifold
equipped with a cooriented contact structure.
We denote by $\D(V, \xi)$ the group of compactly supported diffeomorphisms of $V$
preserving $\xi$, equipped with the strong $C^\infty$-topology. This paper is about the
connected component  $\Do(V, \xi)$ of the identity in $\D(V, \xi)$, and about its
universal cover $\tildeDo{V, \xi}$. See the introduction of
\cite{Fabio_vrille} for information about the complementary question of
studying the mapping class group $\D(V, \xi)/\Do(V, \xi)$.

Klein's Erlangen program suggests to study $(V, \xi)$ through its automorphism
group, which could be any of the above three groups. Nothing is lost by this
perspective according to
\cite{Banyaga_McInerney} (see also \cite[Section 7.5]{Banyaga_book}) which
proved that every group isomorphism $\Phi \co \Do(V_1, \xi_1) \to \Do(V_2, \xi_2)$ (not
necessarily continuous) is induced by a contact isomorphism: there exists a diffeomorphism
$\varphi \co V_1 \to V_2$ such that $\varphi_*\xi_1 = \xi_2$ and $\Phi(g) = \varphi g\varphi^{-1}$ for every $g$ in 
$\Do(V_1, \xi_1)$.

The main known result about the algebraic structure of these groups is proved
in \cite{Rybicki}: both $\Do(V, \xi)$ and its universal cover are \emph{perfect}
groups (every element is a product of commutators), see also
\cite{Tsuboi_contact_simp} for the non-smooth case. Combined with the results of
\cite{Epstein_simplicity} this implies that $\Do(V, \xi)$ is \emph{simple}
(any non-trivial normal subgroup is the full group), see \cite[Theorem 2]{Banyaga_McInerney} or
\cite[Corollary 1.2]{Rybicki}. Note that $\tildeDo{V, \xi}$ is not simple in general
due to the exact sequence:
\[1\to \pi_1 \Do(V, \xi) \to \tildeDo{V, \xi} \to \Do(V, \xi) \to 1.\] 

Seemingly independently of this algebraic structure studies, one can seek an
interesting geometry on these groups. Inspired by the Hofer and Viterbo
distances in symplectic geometry, there have been several recent papers on
invariant norms on contact transformation groups, see
\cite{Sandon_norm, FRP, ColinSandon, Zapolsky, GKPS}, and the survey
\cite{Sandon_survey}.
A conjugation invariant norm on a group $G$ is a function $\nu \co G \to [0, \infty)$
satisfying the following properties:
\begin{itemize}
  \item $\nu(\Id) = 0$ and $\nu(g) > 0$ for all $g \neq \Id$.
  \item $\nu(gh) \leqslant \nu(g) + \nu(h)$ for all $g, h \in G$.
  \item $\nu(g^{-1}) = \nu(g)$ for all $g \in G$.
  \item $\nu(hgh^{-1}) = \nu (g)$ for all $g, h \in G$.
\end{itemize}
Bi-invariant distances are another point of view on the same objects. Such a
distance $d$ defines a norm $\nu = d(\cdot, \Id)$ and, starting from a norm $\nu$,
one gets a distance $d(f, g) = \nu(fg^{-1})$. Invariant norms can arise from
quasi-morphisms, see \cite[Section~1.1]{Bavard}.

Given that the groups we consider are huge (they remember everything about the
contact manifold), such a geometric structure is disappointing if the norm is
bounded or, equivalently, if the associated metric space has finite diameter.
\cite{FRP} proved that this always happens for norms on $\Do(\mathbb{S}^{2n+1}, \xi_0)$ 
or its universal cover\footnote{Here and elsewhere, we assume $n \geqslant 1$.}. By
contrast, based on work of \cite{Givental}, \cite{ColinSandon} proved that the universal cover of 
$\Do(\mathbb{R}P^{2n+1}, \xi_0)$ has an unbounded invariant norm.
Similar puzzling differences between spheres and projective spaces were observed
in the study of orderability of the corresponding groups in \cite{EKP}.
A  direct link between invariant norms and orderability is
provided by \cite{FRP} which proved that orderable contact manifolds with a
periodic Reeb flow have an unbounded invariant norm on the universal cover of
$\Do$, and by \cite{ColinSandon} which defines a norm assuming orderability.

\cite[Corollary~1.17]{EKP} explains how to deduce non-orderability of spheres
from the existence of a 2-subcritical Weinstein filling $W$. But the proof uses
only the existence of a  splitting $W = W' \times \mathbb{C}^2$, not the
flexibility of $W$, and not directly the flexibility of the relevant Legendrian
submanifolds (attaching spheres or stable manifolds). Hence one can argue that
there was no previously known direct link between Gromov's $h$-principle
dichotomy and the existence of unbounded norms on transformation groups.

Before describing such a link, we make contact with the algebraic discussion of
simplicity. A group being simple means concretely that for every element $f$ and
every non-trivial element $g$, one can write $f$ as a product of $N(f, g)$
conjugates $h_igh_i^{-1}$ or $h_ig^{-1}h_i^{-1}$, for some (finite) number 
$N(f, g)$. Independently of simplicity, if this holds for some $f$ and $g$ then,
by definition of invariant norms, $\nu(f) \leqslant N(f, g) \nu(g)$ for any
invariant norm $\nu$. Hence, invariant norms are all bounded as soon as there is
some $g$ such that $N(f, g)$ can be bounded independently of $f$.  A group is called
\emph{uniformly simple} if the number $N(f, g)$ can be bounded independently of
$f$ and $g$.

\begin{theoremintro}\label{thm:main}
  If a closed connected contact manifold $(V, \xi)$ has a supporting open book
  whose pages are flexible Weinstein manifolds, then for every non-trivial 
  $g \in \Do(V,\xi)$, every other element is a product of at most 
  $128(\dim V + 1)$ conjugates of $g^{\pm 1}$. The same holds for the universal
  cover $\tildeDo{V, \xi}$ as soon as $g$ does not lie above the identity. In
  particular, $\Do(V,\xi)$ is uniformly simple and all invariant norms on
  $\Do(V,\xi)$ or $\tildeDo{V, \xi}$ are bounded.
\end{theoremintro}

The relevant definitions will be recalled in \cref{sec:openbook}, including the
special case of $4$-dimensional pages (see \cref{rem:flexible4}).
Note that using \cite{Tsuboi_contact_simp} instead of \cite{Rybicki}
allows to prove the same result about the group of $C^r$ contact transformations
with $1 \leqslant r < 1+ \dim(V)/2$.

The proof of \cref{thm:main} explicitly goes through flexibility of loose
Legendrian embeddings. From the point of view of this theorem, the observed
difference between spheres and projective spaces is not so much related to
fillings, that are external to the contact manifold at hand, but rather to pages
of open books, that are internal symplectic manifolds. The standard contact
structure on $\mathbb{S}^{2n+1}$ has a well known open book whose pages are
$\mathbb{C}^n$, the most extreme example of a flexible Weinstein manifold (hence
\cref{thm:main} includes the result of \cite{FRP} about spheres). On the other
hand, the most well known open book supporting the standard contact structure on
$\mathbb{R}P^{2n+1}$ has page $T^*\mathbb{R}P^n$. 
Since we don't know any other way to obstruct existence of
supporting open books with flexible pages, we state the following corollary
(where projective space could be replaced by any manifold where an unbounded
norm is known to exist).

\begin{corollaryintro}
  The standard contact structure on projective space has no supporting open book
  with flexible Weinstein pages.
\end{corollaryintro}

The known examples of contact manifolds satisfying the hypotheses of
\cref{thm:main} are obtained as the ideal contact boundary of 
$W\times \mathbb{C}$ for a flexible Weinstein manifold $W$ (this corresponds to
the case where the monodromy is the identity). This includes the class of
contact manifolds that \cite[Corollary~1.17]{EKP} prove to be non-orderable.
Hence one can ask whether this result extends to all contact manifolds supported
by open books with flexible pages.

The proof of \cref{thm:main} goes through the following result, where the
dimension of the manifold does not appear, and which allows to compute the
Colin-Sandon norms.

\begin{theoremintro}\label{thm:trente_deux_conj}
  Let $(V, \xi)$ be a closed connected contact manifold. Let $\psi_t$ be a positive (or
  negative) contact isotopy. If $\xi$ is supported by an open book with flexible
  page, and $\epsilon > 0$ is small enough, then every element of $\Do(V, \xi)$ or its
  universal cover is a product of at most 32 conjugates of $\psi_\epsilon^{\pm1}$. 
  In particular, the oscillation and discriminant (pseudo-)norms of Colin and
  Sandon are bounded by $32$.
\end{theoremintro}

Another important point is that the definition of invariant norms asserts no
relation between $\nu$ and any a priori given topology on the group. The norm can
be defined purely in the algebraic world. A popular algebraic example is the
commutator length on perfect groups, which is defined as the minimal number of
factors required to express an element as a product of commutators (under the
flexible page assumption, the proof of \cref{thm:trente_deux_conj} will show
that eight factors is always enough). On the symplectic side, the topology
induced by the Hofer distance is very different from the smooth topology, and
related to the $C^0$ topology. Here we prove the following general result,
without any open book assumption.

\begin{theoremintro}\label{thm:c0}
  Let $(V, \xi)$ be any closed connected contact manifold. Let $\psi_t$ be a
  positive or negative contact isotopy. For $\epsilon > 0$ small enough, there
  exists a $C^0$ neighborhood $\mathcal{U}$ of $\Id$ in $\Do(V, \xi)$ or its
  universal cover such that every element of $\mathcal{U}$ is a product of at
  most 16 conjugates of $\psi_\epsilon^{\pm1}$. In particular all invariant norms are
  bounded on $\mathcal{U}$.
\end{theoremintro}

In the above result, belonging to a $C^0$ neighborhood of identity in the
universal cover means being represented by a path which is $C^0$-close to
identity for all time. A priori this is more restrictive than the pull back of
$C^0$ topology from $\Do(V, \xi)$.

The proofs of all these results rely on geometric decompositions of contact
transformations following the strategies of 
\cite{BIP, Tsuboi_perfectness, Tsuboi_simplicity} which proved analogous
theorems for diffeomorphism groups (without any flexibility assumption).
The decomposition statement is in terms of Giroux's contact handlebodies, ie
open contact manifolds that retract by contact isotopy into arbitrarily small
neighborhoods of isotropic complexes. The relevance of such kinds of compressions
was explicitly pointed out by \cite{BIP} and already imported into contact
topology by \cite{FRP}, with a seemingly more general definition of so called
portable contact manifolds, but it seems that contact handlebodies are the only
known examples.

\begin{theoremintro}
  \label{thm:decomp_intro}
  Let $(M, \xi)$ be a closed connected contact manifold supported by an open
  book with flexible pages. Every contact isotopy is homotopic with fixed
  end-points to the composition of four contact isotopies with compact
  support in the interior of contact handlebodies.
\end{theoremintro}

\paragraph{Outline}

\cref{sec:algebra} is an expository section recalling elementary but fundamental
tools in the algebraic study of transformation groups, with special care devoted
to the universal cover case. 
\cref{sec:loose} gathers the statements we need from Murphy's study of
Legendrian flexibility, and carefully proves a folklore stability result: the
symplectization of a loose Legendrian embedding is loose in the contactization
of the symplectization of the ambient manifold.
\cref{sec:openbook} recalls somewhat under-documented aspects of Giroux's theory
of convexity in contact topology and open book decompositions, relating
\cite{Giroux_these} and \cite{Giroux_ICM}.
\cref{sec:transversality_transfo} proves some general transversality
theorem for multiple jets of families of contact transformations.
\cref{sec:cleaning} uses this general result to show that generic contact
isotopies satisfy a list of conditions that are helpful in the proof of the above
decomposition theorem.
\cref{sec:decomposition} proves the decomposition theorem by reduction to
the generic case and using Murphy's result.
\cref{sec:commutators_and_fragments} combines the decomposition theorem,
algebraic tools from \cref{sec:algebra} and Rybicki's theorem to prove the main
results (Giroux's existence of supporting open books is also used for \cref{thm:c0}).

\begin{acknowledgments}
  We thank Emmy Murphy for interesting conversations about her flexibility
  result, and about looseness of cores of flexible handles. We thank Frédéric
  Le Roux for bringing Tsuboi's work on diffeomorphism groups to our attention
  at a time when we knew only about the three dimensional case handled in
  \cite{BIP}. We thank Tomasz Rybicki for illuminating explanations about his
	perfectness proof, especially Lemma 8.6 from \cite{Rybicki}.
  The first author would like to thank Stéphane Guillermou for many
  stimulating discussions, especially about \cref{prop:displacing}.

  This work was partially funded by ANR grant Microlocal ANR-15-CE40-0007.
\end{acknowledgments}

\section{Displacement, compression and conjugates}
\label{sec:algebra}

In this mostly expository section, we review definitions and
observations in the algebraic study of transformation groups. These
observations are all elaborations on the fundamental observation that
transformations with disjoint supports commute. They originate at least as far as
\cite{Anderson} and play a key role in \cite{BIP, Tsuboi_perfectness,
Tsuboi_simplicity, Tsuboi_even}, which are our main sources for this section.
In the following exposition, we will pay somewhat more attention to the
universal cover of the relevant group (expliciting uses of
\cref{lem:conj_homotopy} below), and bypass some technical definitions
that are not necessary for us.

In this section we fix a smooth manifold $M$ and a connected and locally
contractible subgroup $G$ of the group $\Dc(M)$ of smooth compactly supported
diffeomorphisms of $M$. Other regularity classes and weaker topological
assumptions would work as well, and we will only use $G = \Do(V; \xi)$ in later
sections, but we want to emphasize that the current section involves no contact
geometry. An \emph{isotopy} in $G$ is a smooth path $t \mapsto \varphi_t$
starting from $\Id$ in $G$, where smooth means that $(t, x) \mapsto \varphi_t(x)$ is
smooth. We denote by $\tildeG$ the space of smooth homotopy classes of
isotopies in $G$, with fixed end-points.
This means that two isotopies $\varphi^0$ and $\varphi^1$ are homotopic if there is a smooth
map $\Phi \co M \times [0, 1] \times [0, 1] \to M$ such that, for all $t$, $s$ and $x$,
$\Phi(x, t, 0) = \varphi^0_t(x)$, $\Phi(x, t, 1) = \varphi^1_t(x)$, $\Phi(x, 0, s) = x$,
$\Phi(x, 1, s) = \varphi^0_1(x) = \varphi^1_1(x)$. In particular this implies $\varphi^0_1 = \varphi^1_1$ and this
common value provides a map $\pi \co \tildeG \to G$ which is a universal
cover for $G$.

Time-wise composition gives a group law on the set of isotopies which descends
to a group law on $\tildeG$ such that $\pi$ is a group morphism.
Time reparametrization of isotopies act trivially on $\tildeG$ hence one
can also define the group law on $\tildeG$ by concatenation of isotopies and
suitable time reparametrization.

A subset $A \subset M$ is \emph{displaced} by a diffeomorphism $g$ if $g(A) \cap A = \emptyset$.
The \emph{support} $\supp(g)$ of an element $g$ in $G$ is the closure of the
set of points $x$ in $M$ that are displaced by $g$. We will also write,
somewhat abusing terminology, that an element $g$ of $\tildeG$ is supported in
some subset $U$ if it can be represented by an isotopy $g_t$ such that 
$\supp g_t \subset U$ for all $t$.

We say that a \emph{flow} $\varphi$ in $G$, i.e. a group homomorphism $t \mapsto \varphi_t$ from
$\mathbb{R}$ to $G$, \emph{compresses} an open set $M'$ onto a subset $L \subset M$ if, for
every neighborhood $U$ of $L$ and every compact $K \subset M'$, there is some $T$
such that $\varphi_t(K) \subset U$ for all $t \geqslant T$.

The \emph{conjugate} of an element $h$ of $G$ or $\tildeG$ by another element
$g$ is $\conj{g}{h} = ghg^{-1}$. It is seen as ``$h$ transported by $g$''. In
particular, $\supp(\conj{g}{h}) = g(\supp h)$.
Our first \lcnamecref{lem:conj_homotopy} will be useful to study $\tildeG$.

\begin{lemma}[Homotopies for conjugates and commutators]
  \label{lem:conj_homotopy}
  Let $f$ and $g$ be two isotopies in $G$. The following are isotopies in $G$
  which are homotopic:
  \[
    \big(t \mapsto \conj{f_t}{g_t}\big) \sim \big(t \mapsto \conj{f_1}{g_t}\big).
  \]
  The same is true with:
  \[
    \big(t \mapsto [f_t, g_t]\big) \sim
    \big(t \mapsto [f_1, g_t]\big) \sim
    \big(t \mapsto [f_t, g_1]\big).
  \]
\end{lemma}

\begin{proof}
  The second part of the statement follows from the first one since
  $[f, g] = \conj{f}{g}g^{-1} = f\conj{g}{f^{-1}}$.

  In order to prove the first part, first notice that all paths indeed start at
  $\Id$. One possible homotopy between the two isotopies is:
  \[
    (s, t) \mapsto \conj{f_{s + (1-s)t}}{g_t}.
  \]
  All required properties can be checked directly. For instance, for all $s$,
  $(s, 0) \mapsto \conj{f_s}{\Id} = \Id$ while $(s, 1) \mapsto \conj{f_1}{g_1}$ which is
  indeed the common end-point of both isotopies.
\end{proof}

The key to uniform simplicity is a result relating displacement and compression
to conjugation.

\begin{proposition}
  \label{prop:huit_conj}
  Let $M'$ and $M''$ be open sets in $M$, $M'' \subset M'$, let $L$ be a compact
  subset in $M''$, and let $g$ be an element of $G$ or $\tildeG$ such that
  $g(L) \subset M'' \setminus L$ and $L \subset g(M'')$. If there exists flows $\varphi$ and $\theta$ in $G$ 
  compressing $M'$ and $M''$ respectively onto $L$, and
  such that $\theta$ has compact support in $M'$, then every element $f$ of $G$ or
  $\tildeG$ that is a product of commutators of elements with support in $M'$
  is a product of at most eight conjugates of $g^{\pm1}$.
\end{proposition}

Everything else in this section consists of internal details of the proof.
The first magic trick turns commutators into conjugates of a displacing
isotopy. It is an easy adaptation to $\tildeG$ of
\cite[Remark~6.6]{Tsuboi_perfectness}. Very close considerations also appear in
\cite[Section~2]{BIP}.

\begin{lemma}[Commutators trading]
  \label{lem:trading}
  Let $a$, $b$ and $g$ be three isotopies in $G$. If $\supp(a_t) \cap g_1(\supp(b_t))$
  is empty then $[a, b]$ is homotopic, with fixed end-points, to a product of
  two conjugates of $g$ and two conjugates of $g^{-1}$.
  In particular $[a_1, b_1]$ is a product of two conjugates of $g_1$ and two
  conjugates of $g_1^{-1}$.
\end{lemma}

\begin{proof}
  We set $\alpha_t = \conj{g_1^{-1}}{a_t}$. By assumption and transport of support,
  $b$ and $c$ have disjoint support hence $\alpha_t b_t = b_t \alpha_t$ and:
  \begin{align*}
    a_tb_ta_t^{-1}b_t^{-1} &= (g_1\alpha_tg_1^{-1}) b_t (g_1\alpha_t^{-1} g_1^{-1}) b_t^{-1} \\
                           &= g_1\alpha_tg_1^{-1} (\alpha_t^{-1} \underbrace{\alpha_t) b_t}_{= b_t \alpha_t}g_1\alpha_t^{-1} (b_t^{-1} b_t) g_1^{-1} b_t^{-1} \\
                           &= g_1\conj{\alpha_t}{g_1^{-1}} \conj{b_t\alpha_t}{g_1} \conj{b_t}{g_1^{-1}}.
  \end{align*}
  Hence $[a_t, b_t]$ is homotopic, through
  \[
    (s, t) \mapsto g_{1+s(t-1)}\conj{\alpha_t}{g_{1+s(t-1)}^{-1}} \conj{b_t\alpha_t}{g_{1+s(t-1)}}
    \conj{b_t}{g_{1+s(t-1)}^{-1}},
  \]
  to the isotopy
  \[
    t \mapsto g_t\conj{\alpha_t}{g_t^{-1}} \conj{b_t\alpha_t}{g_t} \conj{b_t}{g_t^{-1}}.
  \]
  and we get the announced decomposition in $\tildeG$.
\end{proof}

The next step is to explain, still following \cite{BIP, Tsuboi_perfectness}, a
sufficient condition allowing to turn a product of any number of commutators
into a product of two commutators.

\begin{lemma}[Commutator crunching]
  \label{lem:crunching}
  Let $U$ be a subset of $M$. Assume there exists an element $\varphi \in G$
  such that the subsets $\varphi^i(U)$ for $i \geqslant 0$ are pairwise disjoint.
  Then any product of commutators of elements of $G$ or $\tildeG$
  with support in $U$ is a product of two commutators.
\end{lemma}

\begin{proof}
  We prove it for $\tildeG$, the case of $G$ is the same except that the
  invocation of \cref{lem:conj_homotopy} at the end is not needed.
  Let $\varphi_t$ be path between $\Id$ and $\varphi$ in $G$.
  Let $f$ be an isotopy in $G$. Assume there is some integer $N$ and isotopies
  $a_i$, $b_i$, in $G$, supported in $U$, $1 \leqslant i \leqslant N$, such that, for all $t$,
  \[
    f_t = \prod_{i = 1}^N [a_{i, t}, b_{i, t}].
  \]
  The goal is to prove that $f$ is a product of two commutators. We set
  \[
    f_{i, t} := \prod_{j = 1}^i [a_{j, t}, b_{j, t}]
  \]
  so $f_0 = \Id$, $f_1 = [a_1, b_1]$, \dots, $f_N = f$.
  Consider the isotopy
  \[
    F_t = \prod_{i = 1}^N \conj{\varphi_1^{N-i}}{f_{i, t}}
  \]
  where each $f_i$ is transported inside $\varphi_1^{N-i}(U)$, as shown in the first
  line of \cref{fig:crunch_commutators}. In this picture and the following
  computations, we drop the subscript $t$ for clarity.
  \begin{figure}[ht]
    \centering
    \newcommand{\nd}[1]{\node[draw, circle, minimum size=1.5cm]{#1};}
    \begin{tikzpicture}
      \matrix [matrix of math nodes,row sep=.5cm, column sep=.5cm]
      {
        & U & \varphi_1(U) & \cdots & \varphi_1^{N - 1}(U) & \varphi_1^N(U) \\
        F & \nd{f = f_N} & \nd{f_{N-1}} & \cdots & \nd{f_1} & \nd{f_0 = \Id} \\
        \conj{\varphi_1}{F} & \nd{\Id} & \nd{f_N} & \cdots & \nd{f_2} & \nd{f_1} \\
        F^{-1}\conj{\varphi_1}{F} & \nd{f^{-1}} & \nd{[a_N, b_N]} & \cdots & \nd{[a_2, b_2]} & \nd{[a_1, b_1]} \\
      };
    \end{tikzpicture}
    \caption{Proof of \cref{lem:crunching}}
    \label{fig:crunch_commutators}
  \end{figure}
  Next consider $\conj{\varphi_1}{F}$ where each piece is shifted by one iterate of
  $\varphi_1$ (second line in \cref{fig:crunch_commutators}). Hence, in $[F^{-1}, \varphi_1] =
  F^{-1}\conj{\varphi_1}{F}$ we get $f^{-1}$ in $U$ and exactly one commutator $[a_i, b_i]$
  in each copy of $U$ since $f_{k-1}^{-1} f_k = [a_k, b_k]$ (third line in
  \cref{fig:crunch_commutators}). In formula:
  \begin{align*}
    [F^{-1}, \varphi_1] &= f^{-1}\prod_{j = 0}^{N-1} \conj{\varphi_1^{N-j}}{[a_{j+1}, b_{j+1}]}\\
                    &= f^{-1}\prod_{j = 0}^{N-1} [\conj{\varphi_1^{N-j}}{a_{j+1}},
    \conj{\varphi_1^{N-j}}{b_{j+1}}].
  \end{align*}
  In the above product, each term has support in its own copy of $U$ hence we
  can rewrite by commutation:
  \[
    [F^{-1}, \varphi_1] = f^{-1}
    \left[\underbrace{\prod_{j = 0}^{N-1}\conj{\varphi_1^{N-j}}{a_{j+1}}}_{=:A},
    \underbrace{\prod_{j = 0}^{N-1}\conj{\varphi_1^{N-j}}{b_{j+1}}}_{=:B}\right]
  \]
  and, reintroducing $t$, we get the final formula $f_t = [A_t, B_t][\varphi_1, F_t^{-1}]$.
  The homotopies from the second part of \cref{lem:conj_homotopy} finish the proof.
\end{proof}

In the above result, the hypothesis of disjointness of the $\varphi^i(U)$ is
easier to check than it may seem. The following lemma ensures it.

\begin{lemma}[\cite{BIP}]
  \label{lem:disjoint_iterates}
  Let $\varphi$ be a transformation of a set $X$. Let $U$ and $W$ be
  disjoint subsets of $X$. If $\varphi(U \cup W ) \subset W$ then all iterates
  $\varphi^i(U)$, $i \geqslant 0$ are pairwise disjoint. 
\end{lemma}

\begin{proof}
  From $\varphi(U \cup W ) \subset W$, we learn that $\varphi^k(U) \subset W$ and 
  $\varphi^{k+l}(W) \subset \varphi^k(W)$, for all $k \geqslant 1$ and $l \geqslant 0$.
  First note that $U \cap W = \emptyset$ whereas all $\varphi^i(U)$, $i > 0$ are
  contained in $W$, hence we can assume $i \geqslant 1$ when proving the lemma.
  Next note that $\varphi^i(U) \subset \varphi^{i-1}(W) \setminus \varphi^i(W)$, for every $i \geqslant 1$.
  Indeed $\varphi^i(U) = \varphi^{i-1}(\varphi(U)) \subset \varphi^{i-1}(W)$,
  but $U$ and $W$ are disjoint hence $\varphi^i(U)$ and $\varphi^i(W)$ are
  disjoint. Since $\varphi^{i+p}(W) \subset \varphi^i(W)$, this can be improved to
  $\varphi^i(U) \subset \varphi^{i-1}(W) \setminus \varphi^{i+p}(W)$ for all $p \geqslant 0$. In particular
  $\varphi^i(U) \subset \varphi^{i-1}(W) \setminus \varphi^{i+p+1}(U)$ for all $p \geqslant 0$, so 
  $\varphi^i(U) \cap \varphi^j(U) = \emptyset$ for all $j > i$.
\end{proof}

We are now ready to prove the main result of this section.
\begin{proof}[Proof of \cref{prop:huit_conj}]
Let $U \subset M''$ be a compact neighborhood of $L$ small enough to ensure 
$W := g(U) \subset M'' \setminus U$ and $U \subset g(M'')$. The conjugated flow
$\bar \theta = \conj{g}{\theta}$ compresses $g(M'')$ into $g(L)$. We fix $T$ such
that $\bar \theta_T(U \cup W) \subset W$. By \cref{lem:disjoint_iterates}, all
iterates $\bar \theta_T^i(U)$, $i \geqslant 0$, are pairwise disjoint in $g(M'')$.

Let $f$ be any element of $G$ or $\tildeG$ which is a product of
commutators of elements with compact support in $M'$.
Up to conjugation by some $\varphi_t$, we can assume these elements have support in $U$,
hence in $g(M')$. \Cref{lem:crunching} then proves that $f$ is a product
of two commutators of elements with compact support in $g(M')$. After
conjugating by $g^{-1}$ and then by some $\varphi_t$, we get elements with
support in $U$. \Cref{lem:trading} then finishes the proof since $g$
displaces $U$.
\end{proof}

\section{Loose Legendrian submanifolds}
\label{sec:loose}

\subsection{Loose charts}

In this section, we recall the main definitions from \cite{Murphy_loose} that
we will need.
On $\mathbb{R}^3$, we consider the contact form $\alpha_3 = dz - pdq$, its Reeb vector field
$\partial_z$, and the front projection $(p, q, z) \mapsto (q, z)$. 
A Legendrian stabilization is a Legendrian arc $\gamma$ in $\mathbb{R}^3$ 
whose front projection has a single transverse self intersection, a single cusp
singularity, and a single Reeb chord. The action of a stabilization is the
action of its Reeb chord.

Let $n\geq 2$ and equip $\mathbb{R}^3 \times \mathbb{R}^{2n-2}$ with the contact form $\alpha = \alpha_3 - \sum y_i dx_i$.
For any contact manifold $(V, \xi)$ and any subset $U \subset V$,
a contact embedding $\varphi \co U \hookrightarrow \mathbb{R}^3 \times  \mathbb{R}^{2n-2}$ is a \emph{loose chart} for
a Legendrian submanifold $L \subset V$ if there is a convex
ball $B \subset \mathbb{R}^3$, and a stabilization $\gamma$ with action $a$ contained in $B$, such that
$(\varphi(U), \varphi(L)) = (B \times [-\rho, \rho]^{2n-2}, \gamma \times {0} \times [-\rho, \rho]^{n-1})$ and
$a/\rho^2 < 2$.

\begin{definition}\label{def:loose}
A Legendrian submanifold $L$ of dimension at least $2$ is called \emph{loose} if
for every connected component $\Lambda$ of $L$, $\Lambda$ admits a loose chart in
$V\setminus (L\setminus \Lambda)$.
\end{definition}

\begin{remark}\label{rem:stabilized}
We reserve the word loose for Legendrian submanifolds of dimension at least $2$.
However the above definition also makes sense for $n=1$, with no quantitative
condition (the constant $\rho$ disappears). We will use the word \emph{stabilized}
for the corresponding condition on $1$-dimensional Legendrian submanifolds.
\end{remark}

\begin{remark}\label{rem:stabilization}
Any connected Legendrian submanifold can be made loose by performing
a smooth (not Legendrian) isotopy supported in an arbitrary neighborhood
of one of its point.
\end{remark}

\subsection{Stability of Loose Legendrian embeddings}

Given a manifold $V$ with a cooriented contact structure $\xi$,
its \emph{symplectization} $SV$ is the submanifold of covectors
in $T^*V$ which define $\xi$, with its coorientation.
Given a Legendrian submanifold $\Lambda \subset V$, its preimage
under the projection $SV \to V$ is an exact Lagrangian
submanifold $S \Lambda \subset S V$. Given a manifold
$B$ with a $1$-form $\lambda$ such that $\d \lambda$ is symplectic,
its \emph{contactization} is the manifold $CB = B \times \mathbb R$
equipped with the contact form $\lambda + \d t$. An exact Lagrangian
submanifold $i:L\to B$ can be lifted to a Legendrian submanifold
$CL \subset CB$ as the graph of a primitive of $-i^* \lambda$ (if $L$ is
connected, such a Legendrian lift is unique up to a translation
in the $\mathbb R$ direction).
One may repeat these operations and consider for example
the Legendrian submanifold $CS\Lambda$ of $CSV$.

\begin{proposition}\label{prop:loose_stability}
  If $\Lambda$ is a loose (or stabilized if $\dim \Lambda =1$) Legendrian submanifold
  in a contact manifold $(V,\xi)$ then $CS \Lambda$ is loose in $CSV$.
\end{proposition}

Related observations appear in \cite[Lemma~3.5]{Murphy_Siegel} and
\cite[Propositions~4.3 and~4.4]{Eliashberg_revisited}.
To prove \cref{prop:loose_stability}, we shall use the following three lemmas.

\begin{lemma}
  \label{lem:reeb_rescaling}
  Let $(V, \xi = \ker \alpha)$ be a contact manifold. Assume that the Reeb flow of
  $\alpha$ is complete.
  Let $\lambda$ be a Liouville form on a manifold $B$. For any function
  $f$ from $B$ to $\mathbb{R}$, there is an isotopy $\varphi$ of $V \times B$ which moves in the
  Reeb flow direction, is relative to $V \times \{ f = 0 \}$, and such that
  $\varphi_t^*(\alpha + \lambda) = \alpha + \lambda + tdf$.
\end{lemma}

\begin{proof}
  Let $R$ be the Reeb vector field of $\alpha$. Let $\varphi$ be the flow of $fR$ on
  $V \times B$, this flow is complete by assumption on $R$.
  The announced formula follows from $d(\varphi_t^*\alpha)/dt = \varphi_t^*\Lie_{fR} \alpha = df$.
\end{proof}

\begin{lemma}
  \label{lem:liouville_rescaling}
  Let $\lambda$ be a Liouville form with complete Liouville flow on a manifold $B$
  and let $(V, \xi)$ be a contact manifold. Given two contact forms $\alpha$
  and $\alpha'$ for $\xi$, there is an isotopy $\varphi$ of $V \times B$ which
  moves in the Liouville flow direction, is relative to the set where $\alpha' = \alpha$, 
  and such that $\ker(\varphi_1^* (\alpha+\lambda)) = \ker(\alpha' + \lambda)$.
\end{lemma}

\begin{proof}
  Let $f = \alpha'/\alpha : V \to \mathbb R_+$ and $X$ be the Liouville
  vector field of $\lambda$. We set $g = \ln(1/f)$.
  Since $\lambda(X) = 0$ and $X \intprod d\lambda = \lambda$, $\Lie_{gX}\lambda = g\lambda$.
  Let $\psi$ be the flow of $X$ on $B$.
  Let $\varphi$ be the flow of $gX$ on $V \times B$: $\varphi_t(v, b) = (v, \psi_{tg(v)}(b))$.
  We have $d/dt(\varphi_t^*(\alpha + \lambda)) = \varphi_t^*(g\lambda) = ge^{tg}\lambda$. Hence
  $\varphi_t^*(\alpha + \lambda) = \alpha + e^{tg}\lambda$,
  and $\varphi_1^*(\alpha + \lambda) = \ker(\alpha + 1/f\lambda) = \ker(\alpha' + \lambda)$.
\end{proof}

In the preceding lemma, the completeness assumption on $\lambda$ is crucial. Indeed,
the discussion of large neighborhoods of contact submanifolds in
\cite{Niederkruger_neighborhood} proves that, when $W = \mathbb{D}^2$ and $\lambda = r^2d\theta$, even
a constant rescaling of a contact form is enough to get different contact
structures on $V \times B$. By contrast, the next lemma states that the completeness
assumption in \cref{lem:reeb_rescaling} does not cost much.

\begin{lemma}\label{lem:reeb_complete}
  Any contact manifold $(V,\xi)$ admits a complete Reeb vector field.
\end{lemma}
\begin{proof}
  Let $f:V\to \R$ be a proper and positive function, and $\alpha$
  a contact form for $\xi$. Let $g = df(R_\alpha)$ and pick a function
  $\psi : (0,+\infty) \to (0,+\infty)$ decreasing sufficiently fast so that
  $g \psi\circ f \leq 1$. Then the contact vector field $X$ with
  contact Hamiltonian $\psi \circ f$ with respect to $\alpha$
  writes $X = \psi \circ f R_\alpha + Y$ where $Y \in \xi$ and $\d f(Y) = 0$.
  Hence $df(X) = g\psi\circ f\leq 1$ and $X$ is a complete Reeb vector field.
\end{proof}

\begin{proof}[Proof of \cref{prop:loose_stability}]
  Let $\alpha$ be a contact form for $\xi$ whose Reeb flow is complete
  as provided by lemma \ref{lem:reeb_complete}.
  This induces a trivialization
  $\Phi : V \times \mathbb R \times \mathbb R \to CSV$
  in which the contact structure reads $\ker(e^t \alpha+ \d s)$.
  The diffeomorphism $g \co (v, t, s) \mapsto (v, t, -e^t s)$ is relative
  to $\{s = 0 \}$ and pulls back the above contact structure to
  $\ker(\alpha - sdt - ds)$. \Cref{lem:reeb_rescaling} then gives a
  diffeomorphism $\Psi$ relative to $\{ s = 0 \}$
  which pulls back this contact structure to $\ker(\alpha - sdt)$.

  Let $B$ be a closed ball around the origin in $\mathbb{R}^3$,
  $\alpha_3 = dz - pdq$, and $\gamma$ a stabilization in $B$ with action $a$.
  We set $U_\rho = B \times [-\rho, \rho]^{n-1} \times [-\rho, \rho]^{n-1} \subset \mathbb{R}^{2n + 1}$.
  By definition of loose Legendrian embeddings, for each component $\Lambda_0$ of $\Lambda$,
  there are some $\rho$ and $a$ with $a/\rho^2 < 2$, and some contact embedding
  $\iota:(U_\rho, \ker(\alpha_3 -\sum y_i dx_i)) \to (V,\xi)$ such that
  $\iota^{-1}(\Lambda) = \iota^{-1}(\Lambda_0) = \gamma \times {0} \times [-\rho, \rho]^{n-1}$.
  Let $\beta$ be a contact form on $V$ such that $\iota^*\beta = \alpha_3 -\sum y_i dx_i$.

  \Cref{lem:liouville_rescaling} gives a self-diffeomorphism $\Theta$ of
  $U_\rho \times \mathbb{R}^2$ which pulls back $\ker(\alpha - sdt)$ to $\ker(\beta - sdt)$, and is relative
  to $\{ s = 0 \}$ since the relevant Liouville vector field is $s\partial_s$.
  The map $\hat{\iota} = \Phi \circ g \circ \Psi \circ \Theta \circ (\iota \times \Id_ {\mathbb{R}^2}) : U_\rho \times [-\rho,\rho]^2 \to CSV$
  is a loose chart for $CS\Lambda$.
\end{proof}

\subsection{Murphy's \texorpdfstring{$h$}{h}-principle}

Recall that when $M$ and $N$ are manifolds with boundary, an embedding
$k:M\to N$ is called \emph{neat} if $k^{-1}(\del N)=\del M$, and $k$ is transverse
to $\del N$ along $\del M$. A \emph{formal} Legendrian embedding in a contact
manifold $(V,\xi)$ is a couple
$(f,F_s)$ where $f:L\to V$ is an embedding and $F_s:T L \to T V$
is a family of monomorphisms covering $f$ such that $F_0=d f$ and $F_1(T L)$
is Legendrian (i.e. Lagrangian in the contact structure). When $F_s=d f$
for all $s$, it is a \emph{genuine} Legendrian embedding. We will usually
suppress $F_s$ from the notation for shortness.

The following statement is a parametric, relative, and folkloric version of
Murphy's h-principle for Loose Legendrian embeddings.

\begin{theorem}[\cite{Murphy_loose}]\label{thm:murphy}
  Let $(V, \xi)$ be a contact manifold with boundary of dimension at least $5$,
$L$ a connected compact manifold with boundary and $(k_t)_{t\in D^p}:L\to V$ a family
of neat formal Legendrian embeddings such that
  \begin{itemize}
    \item $k_t$ is genuine near $\del L$,
    \item $k_t$ is genuine for $t \in \del D^p$,
    \item there is a fixed loose chart $U \subset \Int (V)$ for all $k_t$,
      i.e. $k_t^{-1}(U) = \Lambda \subset \Int L$ is independent of $t$,
    $k_t$ is independent of $t$ on $\Lambda$ and the pair $(U, k_t(\Lambda))$ is a
    loose chart.
  \end{itemize}
  Then there exists a homotopy $(k_{t,s})_{(t,s)\in D^p \times[0,1]}$ of neat formal
  Legendrian embeddings such that
  \begin{itemize}
    \item $k_{t,0} = k_t$,
    \item $k_{t,s}$ has a fixed loose chart $U' \subset U$,
    \item $k_{t,s}=k_t$ near $\del L$,
    \item $k_{t,s}=k_t$ for $t\in \del D^p$,
    \item $k_{t,1}$ is a genuine Legendrian embedding.
  \end{itemize}
\end{theorem}

\section{Giroux's convexity theory}\label{sec:openbook}

\subsection{Convexity from symplectic to contact}

A Weinstein manifold, as defined in \cite{Eliashberg_Gromov_convex}, is a tuple
$(W, \lambda, f)$ where $\lambda$ is a 1-form whose differential is symplectic and $f$ is an
exhausting Morse function which is Lyapounov for the vector field $Z$ defined
by $Z\lrcorner \d\lambda = \lambda$. 
The union of all stable manifolds of critical points of $f$ form the so-called
\emph{isotropic skeleton} $L$, it is a union of isotropic
submanifolds diffeomorphic to euclidean space, on which $\lambda$ vanishes.
One may lift $L$ to the contactization $(W \times \R, \lambda + d t)$ as $L\times \{0\}$
which is a union of isotropic submanifolds. As soon as $f$ has finitely many critical
points and $Z$ is Morse-Smale, this satisfies the following tameness condition
which is very weak, but sufficient for most of our purposes.

\begin{definition}\label{def:isotropiccomplex}
A subset $L$ of a contact manifold of dimension $2n+1$ is called an \emph{isotropic
complex} if it admits a filtration:
\[\emptyset=L^{(-1)} \subset L^{(0)} \subset \cdots \subset L^{(n)}=L\]
such that for all $i\in\{0,\dots,n\}$, $L^{(i)}\setminus L^{(i-1)}$
is an isotropic submanifold without boundary of dimension $i$,
and near each point of $L^{(n)}\setminus L^{(n-1)}$, $L=L^{(n)}\setminus L^{(n-1)}$.
\end{definition}

Based on Murphy's work, \cite[Definition 11.29]{Cieliebak_Eliashberg_book}
defined the notion of a flexible Weinstein manifold as follows.

\begin{definition}\label{def:flexible}
$(W,\lambda,f)$ is \emph{flexible} if there is an increasing sequence of real numbers
$(c_i)_{i\geq 0}$ such that:
\begin{itemize}
\item $c_0<\min f$,
\item $c_i \to +\infty$,
\item $c_i$ is a regular value of $f$,
\item there are no $Z$-trajectories joining critical points in $\{c_i<f<c_{i+1}\}$,
\item the link of Legendrian spheres in $\{f=c_i\}$ corresponding to attaching spheres of
index $n$ critical points lying in $\{c_i<f<c_{i+1}\}$ is loose (see \cref{def:loose}).
\end{itemize}
\end{definition}

\begin{remark}\label{rem:flexible4}
By extension, we also use the word flexible in the case where
$W$ is $4$-dimensional and the attaching spheres are stabilized, see
\cref{rem:stabilized}. Note however that, although this assumption is enough for
our purposes, it is a priori not enough for the purpose of deforming Weinstein
structures.
\end{remark}

Using \cref{prop:loose_stability} we see that when $W$ is flexible (including
the $4$-dimensional case), the lift of its isotropic skeleton to the contactization
$CW$ is loose in the following sense.

\begin{definition}\label{def:loosecomplex}
An isotropic complex $L$ is \emph{loose} if $L^{(n)}\setminus L^{(n-1)}$
is loose in $V\setminus L^{(n-1)}$.
\end{definition}

\cite{Eliashberg_Gromov_convex} also proposed, as a
contact analogue of Weinstein manifolds, to study contact manifolds equipped
with a Morse function $f$ and a contact pseudo-gradient of $f$. Exploration of
this definition began in \cite{Giroux_these}, and eventually became Giroux's
open book theory. In particular it follows from the existence of supporting open
books that every closed contact manifold is convex, in contrast to the
symplectic case. Since this implication is not documented, and we need a rather
precise statement, we will explain it in this section. 

The following proposition is everything we will need. Recollections about open
books, including the precise meaning of ``exact symplectomorphic to pages'' will
be in \cref{sec:ideal_open_book}.

\begin{proposition}[Giroux $\sim 2001$]
  \label{prop:leg_skeleta}
  Let $(V, \xi)$ be a contact manifold supported by an open book $(K, \theta)$
  with Weinstein pages. Let $(W_\pm, \lambda_\pm, f_\pm)$ be two Weinstein manifolds that
  are exact symplectomorphic to the pages. There exist:
  \begin{itemize}
    \item 
      contact embeddings $j_\pm \co (W_\pm \times \mathbb{R}, \ker(\lambda + \d t)) \hookrightarrow (V \setminus K, \xi)$
      with disjoints images
    \item 
      a Morse function $F \co V \to \mathbb{R}$
    \item
      a pseudo-gradient $X$ for $F$ whose flow preserves $\xi$
  \end{itemize}
  such that, denoting $L_\pm$ the image of the skeleton of $W_\pm$ by the embedding
  $j_\pm$ restricted to $W_\pm\times\{0\}$, we have for every neighborhood $U_\pm$ of
  $L_\pm$ in $V$, the flow $\varphi_t$ of $X$ satisfies:
  \[
    \bigcap_{t \geqslant 0} \varphi_t\left(V \setminus U_-\right) = L_+ \qquad
    \bigcap_{t \leqslant 0} \varphi_t\left(V \setminus U_+\right) = L_-.
  \]
\end{proposition}

The proof of the above proposition actually gives more information about the
relation between $(F, X)$ and the Weinstein data:

\begin{remark}
  \label{rmk:more_on_convex_functions}
  In the context of \cref{prop:leg_skeleta}, let $M$ be a real number
  bigger than all critical values of $f_-$ and $f_+$. One can construct $F$ such
  that:
  \begin{itemize}
    \item
      $F$ and $X$ extend the restrictions of $f_-$ and $Z_-$ to ${\{f_- \leqslant M\}}$ 
    \item
      $F$ and $X$ extend the restrictions of $4M - f_+$ and $-Z_+$ to ${\{f_+ \leqslant M\}}$ 
    \item
      critical points of $F$ are exactly critical points of $f_+$ and $f_-$ ; with
      the same index for points coming from $f_-$, and index augmented by one for
      points coming from $f_+$
    \item 
      $F\rst{K} = 2M$
    \end{itemize}
\end{remark}

\subsection{Contact structures and open book decompositions}
\label{sec:ideal_open_book}
An \emph{open book decomposition} of a manifold $V$ is a codimension
$2$ submanifold $V$ with trivial normal bundle together with
a fibration $\theta : V \setminus K \to \mathbb R /2\pi\mathbb Z$
which corresponds to the angular coordinate of $D^2$ in some tubular
neighborhood $K \times D^2$ of $K$ in $V$. $K$ is called the \emph{binding}
and the closure of the fibers of $\theta$ are called the \emph{pages}.
The pages are always cooriented using the canonical orientation
on $\mathbb{R}/2\pi \mathbb{Z}$. If $V$ is oriented, this orients the pages
and the binding is then oriented as the boundary of pages.
Recall the following important definition due to Giroux which provides
a link between contact structures and open book decompositions.

\begin{definition}\label{def:carried}
  A contact structure $\xi$ on $V$ is \emph{carried} by an open book
  $(K,\theta)$ if there exists a contact form $\alpha$ for $\xi$ such that
  \begin{itemize}
    \item $\alpha$ induces a contact form on $K$,
    \item $d \alpha$ induces a symplectic form on the interior of each page,
    \item the orientation of $K$ induced by $\alpha$
      agrees with the boundary orientation of the pages
      (oriented by $\d \alpha$).
  \end{itemize}
\end{definition}

Let us call any contact form as in definition \ref{def:carried} a
\emph{Giroux form} adapted to $(K,\theta)$. It will be more convenient
here to work with the following definition, which is tiny variation on the setup
discussed in \cite{Giroux_ideal}.

\begin{definition}\label{def:ideal_giroux_form}
  Let $(K, \theta)$ be an open decomposition of a manifold $V$. An
  \emph{ideal Giroux form} adapted to $(K, \theta)$ is a contact form $\alpha$ on $V \setminus K$ such
  that:
  \begin{itemize}
    \item
      $d\theta(R_\alpha) > 0$,
    \item
      there is some positive contact form $\beta$ on $K$ and
      a tubular neighborhood $i:K \times D^2\to V$ of $K$ such that
      $i^*\theta = \varphi$ and $i^*\alpha = d \varphi + \beta/r^2$ where
      $(r,\varphi)$ are polar coordinates on $D^2$.

  \end{itemize}
\end{definition}

The two notions are essentially equivalent according to the following lemma.

\begin{lemma}\label{lem:ideal_giroux}
  Let $V$ be a closed manifold with an open book decomposition $(K,\theta)$
  and $\xi$ a contact structure on $V$.
  \begin{enumerate}
    \item If $\xi$ is defined by an ideal Giroux form adapted to $(K,\theta)$
      then it is carried by $(K,\theta)$.
    \item If $\xi$ is carried by $(K,\theta)$, then after a small
      deformation of $\theta$ near $K$, it is defined by an ideal
      Giroux form adapted to $(K, \theta)$.
  \end{enumerate}
\end{lemma}

\begin{proof}
Let $\alpha$ be an ideal Giroux form adapted to $(K,\theta)$ and $i:K \times D^2 \to V$ a
tubular neighborhood of $K$ as in definition \ref{def:ideal_giroux_form}. We
modify $\alpha = \frac{\beta}{r^2} + \d \theta$ in this neighborhood by replacing it
by $\alpha' = f(r) (\beta + r^2 \d \varphi)$ with $f(r) = \frac{1}{r^2}$ near 
$\partial D^2$. The condition on $f$ for $\alpha'$ to be a Giroux form simply writes $f'<0$
for $r>0$ which can be easily achieved.

Conversely, let $\alpha$ be a Giroux form adapted to $(K,\theta)$. Pick a tubular
neighborhood $i:K \times D^2 \to V$ of $K$ such that $i^* \theta = \varphi$. The
conformal symplectic normal bundle $\nu$ of $K$, i.e. the $\d \alpha$-orthogonal to
$T K$ in $\xi$, inherits a trivialization from $i$. The tubular neighborhood
theorem for contact submanifolds then allows to deform $i$ without changing its
differential along $K \times \{0\}$ to achieve $i^*\alpha = f(\beta+r^2 \d \varphi)$ for
some positive function $f$ and a positive contact form $\beta$ for $K$. One
issue is that $i^*\theta$ is no longer equal to $\varphi$. However, $i^*\theta$
is the angular coordinate of coordinates which only differ by a diffeomorphism
tangent to the identity along $K \times \{0\}$. Hence we may deform $i^* \theta$ near
$K \times\{0\}$ so that it agrees with $\varphi$ near $K \times\{0\}$, simply by replacing
it with $\rho(r)\varphi + (1-\rho(r))i^*\theta$ for some cutoff function
$\rho(r)$ supported in a small neighborhood of $0$. To get an ideal Giroux form
it remains to replace the function $f$ by a function equal to $\frac{1}{r^2}$
near $K \times \{0\}$, equal to $f$ away from a neighborhood of $K \times\{0\}$ and still
satisfying the condition $\frac{\del f}{\del r}<0$ for $r>0$.
\end{proof}

Here are some nice features of ideal Giroux forms.
\begin{itemize}
  \item The contact structure $\ker \alpha$ extends smoothly over $K$ since
    $\ker(d\theta + \beta/r^2) = \ker(\beta + r^2d\theta)$.
  \item The interior of each page equipped with the restriction
    of $\alpha$ is a finite type complete Liouville manifold and its
    contact boundary at infinity is identified with $K$.
  \item The holonomy of the Reeb vector field $R_\alpha$
    defines \emph{exact symplectomorphisms} between the interior of the pages.
    Namely, if $\gamma:[a,b]\to \R/2\pi \Z$ is a path with $\gamma'(t) = 1$,
    then the holonomy map
    $h_\gamma : \theta^{-1}(\gamma(a)) \to \theta^{-1}(\gamma(b))$
    satisfies $h_\gamma^*\alpha = \alpha + d f_\gamma$
    where $f_\gamma(x)$, $x \in \theta^{-1}(\gamma(b))$ is the time that the Reeb flow
    takes to travel from $x$ to $\theta^{-1}(\gamma(b))$
    above $\gamma$. Since $R_\alpha = \del_\theta$ in $K \times D^2$, the function
    $f_\gamma$ is equal to $b-a$ outside of a compact set. In particular, the
    monodromy map $\phi: \theta^{-1}(0) \to \theta^{-1}(0)$
    (corresponding to $a = 0$, $b = 2\pi$) is compactly supported and satisfies
    $\phi^*\alpha = \alpha + \d f$ with $f>0$ and $f = 2\pi$ outside of a compact set.
\end{itemize}
In particular, it makes sense to speak of an open book decomposition
with page a given complete finite type Liouville manifold $(W,\lambda)$.
The following lemma allows to fix specific Liouville forms
(not only up to exact symplectomorphism) on any finite collection of pages,
and to get disjoint copies of contactizations of these Liouville manifolds
inside our contact manifold.

\begin{lemma}\label{lem:reembed_page}
  Let $V$ be a closed manifold equipped with an ideal Giroux form $\alpha$
  adapted to an open book $(K,\theta)$. Let $(W,\lambda)$ be a finite type
  Liouville manifold exact symplectomorphic to the pages. 
  For every $c \in \mathbb{S}^1$, and every open interval $I \subset \mathbb{S}^1$ containing $c$, there is a
  family of fibrations $\theta_s \co V \setminus K \to \mathbb{S}^1$ such that
  \begin{itemize}
    \item $\theta_s = \theta$ near $K$ and $\theta^{-1}(\mathbb{S}^1 \setminus I)$, and everywhere
      when $s = 0$, 
    \item $\alpha$ is adapted to $(K, \theta_s)$ for all $s$,
    \item there is an embedding $g : W \times \mathbb{R} \to \theta^{-1}(I)$ such that
      $g^*\alpha = f(\lambda + d t)$ everywhere, and $f = 1$ and
      $g^* \theta_1 = c + t \pmod{2\pi}$ on $W \times (-\epsilon, \epsilon)$, for some
      positive $\epsilon$.
  \end{itemize}
\end{lemma}

\begin{proof}
  By assumption, there exists a diffeomorphism
  $i \co W \to \theta^{-1}(c)$ such that $i^*\alpha = \lambda + \d k$
  for some compactly supported function $k$.
  This diffeomorphism combines with the flow $\varphi$ of the Reeb field $R$ of $\alpha$ to
  give an immersion $j_0 \co W \times \mathbb{R} \looparrowright V \setminus K$ sending $(x, t)$ to $\varphi_t(i(x))$.
  We compute $j_0^*\alpha = i^*\varphi_t^*\alpha + \alpha(R)dt = \lambda + dk + dt$.
  Precomposing with $(x, t) \mapsto (x, t - k(x))$ gives a new immersion
  $j$ such that $j^*\alpha = \lambda + dt$, but now the page $\theta^{-1}(c)$ is the image of the
  graph of $\{t = k(x) \}$ in $W \times \mathbb{R}$. Let $f_+(x)$ (resp. $f_-(x)$) be the time taken by the
  flow $\varphi$ to travel (resp. travel back) from $i(x)$ to $\theta^{-1}(\mathbb{S}^1 \setminus I)$. Both
  these functions are bounded below by a positive constant thanks to the model
  behavior of $R$ near $K$. The restriction of $j$ to  
  $Y = \{(x, t) \;;\; k(x) - f_-(x) < t < k(x) + f_+(x)\}$ is a diffeomorphism onto
  $\theta^{-1}(I)$.

  What we know about $j^*\theta$ is that its level set $\{j^*\theta = c \}$ is
  $\{t = k(x) \}$. We would like to deform the function $j^*\theta$
  in a compact of $Y$, among increasing functions of $t$, so that we have $j^*
  \theta = c + t \pmod{2\pi}$ wherever $|t| < 2\epsilon$. This is possible if (and only if)
  $k - f_- < 0$ and $k + f_+ > 0$.
  This condition can be guaranteed by the following trick: denoting by $\psi$ and
  $\psi'$ the Liouville flows of $\lambda$ and $\alpha\rst{T\theta^{-1}(c)}$ respectively, replacing
  $i$ by $\psi'_{-s} \circ i \circ \psi_s$ replaces $k$ by $e^{-s}k\circ\psi_s$.
  Since $k$ was already compactly supported, it can be made arbitrarily
  $C^0$-small (note that this trick would not allow to get a $C^1$-small $k$),
  without changing $f_\pm$.

  It remains only to replace $Y$ by the full $W \times \mathbb{R}$ without loosing control of
  $j^*\theta$.
  Let $\rho : \R\to (-2\epsilon, 2\epsilon)$ be diffeomorphism equal to the
  identity on $(-\epsilon, \epsilon)$ and $h = \id \times \rho : W \times \R \to W \times (2\epsilon,
  2\epsilon)$.
  We have $h^*(\lambda + \d t) = \lambda + \rho'(t) \d t$. Since the Liouville flow
  of $\lambda$ is complete, lemma \ref{lem:liouville_rescaling}
  provides a diffeomorphism $\psi:W \times \R \to W \times \R$ fibered over
  $\R$ such that $\ker(\psi^*(\lambda + \rho'(t) \d t)) = \ker(\lambda + \d t)$,
  and $\psi = \Id$ on $W \times (\epsilon, \epsilon)$.
  The embedding $g = j \circ h\circ\psi$ has all the required properties.
\end{proof}

\subsection{From open book decompositions to convex Morse functions}

The goal of this section is to prove \cref{prop:leg_skeleta} and the companion
\cref{rmk:more_on_convex_functions}.

Let $\alpha$ be an ideal Giroux form adapted to $(K, \theta)$.
\Cref{lem:reembed_page} gives disjoint contact embeddings 
$i_\pm \co W_\pm \times \mathbb{R} \hookrightarrow V \setminus K$ such that, on 
$W_\pm \times (-\epsilon, \epsilon)$, $i_-^*\theta = t \pmod{2\pi}$, $i_-^*\theta = \pi + t \pmod{2\pi}$, and
$i_\pm^*\alpha_\pm = \lambda_\pm + \d t$. In particular $i_\pm^*R_\alpha = \partial_t$. 

The starting observation is that, in the contactization of $W_\pm$, the vector
field $\mp(Z_\pm + t\partial_t)$ is contact and pseudo-gradient of $\mp(f_\pm + t^2)$, with the
required dynamics. For the purpose of fitting both contactizations together, it
is more convenient to use the variant $\cos i_\pm^*\theta Z_\pm + \sin i_\pm^*\theta R_\alpha$
(which is the same thing up to order one when $t$ goes to zero, i.e. when $\theta$
is close to $0$ of $\pi$.).

Since $i_\pm^*\alpha(\cos i_\pm^*\theta Z_\pm + \sin i_\pm^*\theta R_\alpha) = \sin i_\pm^*\theta$, we can patch these
two vector fields to the contact vector field $X'$ on $V \setminus K$ with contact
Hamiltonian $\sin \theta$ with respect to $\alpha$ (this is almost the required $X$, but
will need some tweaking to extend over $K$). We have $X' = \sin \theta R_\alpha + Y$
where $Y$ is in $\ker d\sin \theta$, hence in $\ker d\theta$, and in $\xi$. 
In particular $d\theta(X') = \sin \theta d\theta(R_\alpha)$. On the image of $W_\pm$, $Y = \mp Z_\pm$ by
construction. 

We now begin to construct the Morse Lyapounov function $F$. In $i_-(W_- \times (-\epsilon, \epsilon))$, we
set $F_- = f_- + \lambda(1 - \cos \theta)$, for some large positive $\lambda$ to be specified later. This
is a Morse function which extends $f_-$ with critical points set 
$\{(w, 0) \;;\; w \in \crit f_-\}$, and the same Morse index than in $W_-$. We reduce
$\epsilon$ to make sure $\epsilon < \pi/2$, hence $\cos \epsilon > 0$. We then have:
\[
  \crit F_- \subset \{ F_- \leqslant M \} \subset W_- \times (-\epsilon/2, \epsilon/2)
\]
as soon as $\lambda \geqslant (M - \min f_-)/(1 - \cos(\epsilon/2))$ (remember $f_-$ is bounded below
by definition of a Weinstein structure). This has the desired pseudo-gradient
because $dF_-(X') = \cos t\, df_-(Z_-) + \lambda\sin^2 t$ in $W_- \times (-\epsilon, \epsilon)$.
We set $A_- = \{ F_- \leqslant M \}$. Our final $F$ will extend $F_-$ from $A_-$ to $V$.

In $i_+(W_+ \times (-\epsilon, \epsilon)$ we have almost the same situation if we set
$F_+ = 4M - f_+ - \lambda(1 + \cos \theta)$. One difference is that 
$-\cos \theta = -\cos(t + \pi) = \cos t$ so the index of critical points goes up by
one. The required estimates are now
\[
  \crit F_+ \subset \{ F \geqslant 3M \} \subset W_+ \times (-\epsilon/2, \epsilon/2)
\]
which hold whenever $\lambda \geqslant (M - \min f_+)/(1 - \cos(\epsilon/2))$.

Before bridging the gap between $A_-$ and $A_+ := \{F_+ \geqslant 3M\}$, we
need to modify $X'$ near the binding $K$.
In the tubular neighborhood $K \times \mathbb{D}^2$, the contact Hamiltonian of $X$ with
respect to $r^2 \alpha = \beta + r^2d\theta$ is equal to $r^2\sin \theta = r y$ (in particular,
$X$ does not extend smoothly on $K$). We replace this Hamiltonian by $\rho(r)y$
where $\rho$ is a smooth non-decreasing function interpolating between a positive constant
near $0$ and $r$. The corresponding contact vector field is our final $X$. 
It is smooth everywhere (since its Hamiltonian is) and coincides with $X'$ away
from a neighborhood of $K$. A computation shows that:
\[
  X = \frac{y}{2}(\rho - r \rho') R_\beta - \frac{\rho}{2} \cos \theta\del_r +
    \frac{y}{2r^2}(r\rho'+\rho)\del_\theta.
\]
In particular $d\theta(X)$ has the sign of $\sin(\theta)$ on all $V \setminus K$, and
$dx(X) = \cos \theta dr(X) - r\sin \theta d\theta(X) = -(\rho + r\sin^2\theta \rho')/2$ is
negative in $K \times \mathbb{D}^2$. Also $X$ is tangent to $\{ \theta = 0 \} \cup \{ \theta = \pi \}$.

We already saw that $X'$, hence $X$, goes transversely out of $A_-$ and into
$A_+$. The above computations also allow to check that $X$ has no zero outside 
$A_- \cup A_+$ and every point of $\partial A_-$ flows to $\partial A_+$ in finite time.
Hence one can extend $F_-$ and $F_+$ to a Morse function $F$ with no critical point 
outside $A_- \cup A_+$, admitting $X$ as a pseudo-gradient, and such that
that $F = 2M$ on $K \cup \{ \theta = \pm \pi/2\}$.

By construction, stable (resp. unstable) manifolds of critical points belonging to 
$\{ \cos \theta > 0 \}$ (resp. $\{ \cos \theta < 0 \}$) are the corresponding stable
manifolds of $Z_-$ (resp. $Z_+$), so this announced dynamics is ensured.

\subsection{Two lemmas about isotropic complexes}

\begin{proposition}\label{prop:displacing}
Let $V$ be a manifold, $\xi$ a cooriented hyperplane field, $L$ a compact submanifold
with conical singularities (in the sense of \cite{Laudenbach_conic}) whose
strata are
integral submanifolds of $\xi$, and $\phi_t$ an isotopy generated by a vector
field $X_t$ which is positively transverse to $\xi$. Then there exists $\epsilon>0$,
such that $\phi_t(L)\cap L = \emptyset$ for all $t\in]0,\epsilon]$.
\end{proposition}
\begin{proof}
If this is wrong, we find sequences $x_n, y_n \in L$ and $t_n>0$
converging to $0$ such that $\phi_{t_n}(x_n)=y_n$. By compactness of $L$
we may assume that $x_n$ and $y_n$ converge, necessarily to the same point $z\in L$.
In local coordinates centered at $z$, we have $\phi_t(x)=x+t X_0(0)+o(|t,x|)$.
Hence, $\frac{y_n-x_n}{t_n}$ converges to $X_0(0)$ as $n$ goes to $+\infty$.
It is therefore enough to prove that at each point $z\in L$, the subset, denoted
$C_z L$, of $T_z V$ consisting of all accumulation points of sequences $\frac{y_n-x_n}{t_n}$
with $x_n,y_n \in L$ converging to $z$ and $t_n$ converging to $0$, is included in $\xi_z$
(and therefore does not contain $X_0(z)$). Let us now prove this fact, by induction on the
dimension of $L$. If $L$ is a finite number of point, it is clear. Assume we have proved it
when $\dim L < k$ and let $L$ be $k$-dimensional. At a point $z$ in a stratum of
dimension $i>0$, we have locally a product situation
$(V,L)\simeq (D^i\times D^{2n+1-i}, D^i \times L')$ where $L'$ is a $(k-i)$-dimensional
complex, and thus $C_z L= \R^i \times C_z L'$. By induction hypothesis, we have $C_z L' \subset \xi_z$
and thus $C_z L \subset \xi_z$. Now if $z$ is a $0$-dimensional stratum of $L$, we have
a $C^1$-chart $\varphi:(D^{2n+1},0)\to (V,z)$ such that $\varphi^{-1}(L)$
is the cone over a compact submanifold with conical singularities $L' \subset S^{2n}$.
For $x\in \varphi^{-1}(L)\setminus\{0\}$, the ray $\{tx,t\in(0,1)\}$ is contained
in a single stratum of $L$ and is hence tangent to $\xi$ at each point.
Since $\varphi^*\alpha$ is continuous at the origin, we obtain that this ray is
in fact entirely contained in $(\varphi^*\xi)_0$. Hence $\varphi^{-1}(L)$
is included in the hyperplane $(\varphi^*\xi)_0$ and we get $C_z L\subset \xi_z$
\end{proof}

\begin{remark}\label{rem:conicsing}
For any Weinstein manifold $(W,\omega, X)$, according to
\cite[Proposition~12.12]{Cieliebak_Eliashberg_book}, we may deform $X$ near the
critical points so that it
is the gradient vector field with respect to a flat metric there.
When this property is achieved, \cite[Proposition~2]{Laudenbach_conic}
guarantees that
the skeleton is a submanifold with conical singularities. \cref{prop:displacing}
will ensure that, when applying \cref{prop:leg_skeleta}, we may assume that
the skeleta $L_\pm$ are displaced by any small positive contact isotopy.
\end{remark}

Proposition \ref{prop:displacing} does not hold for a general isotropic complex
without some taming condition (in $\R^3$, think of Legendrian curves whose
Lagrangian projection spiral around the origin).

\begin{lemma}\label{lem:frag}
Let $\phi_t$, $t\in[0,1]$, be a contact isotopy of $(V,\xi)$ and $L$ an
isotropic complex such that for each $i\in\{0,\dots,n\}$, there exists a basis
of open neighborhoods $U_i$ of $L^{(i)}$ such that $L^{(i+1)} \setminus U_i$ is
diffeomorphic to a finite disjoint union of disks. Then there exists a
collection of contact isotopies $\phi_{i,t}$ supported in the interior of Darboux
balls $B_i$ such that, for $t$ near $0$, $\phi_t=\phi_{0,t} \circ \cdots  \circ \phi_{n,t}$ near $L$.
\end{lemma}

\begin{proof}
We proceed by induction. By assumption $L^{(0)}$ is a finite union of points and
it is thus contained in the interior of some Darboux ball $B_0$. Multiply the
section of $T V/\xi\to V\times[0,1]$ corresponding to $\phi_t$ by a function equal to
$1$ on some neighborhood $V_0$ of $L^{(0)}$, and supported in the interior of
$B_0$. This generates a contact isotopy $\phi_{0,t}$ supported in $\Int B_0$ such
that $\phi_{0,t}=\phi_t$ near $L^{(0)}$ and for $t \leqslant \epsilon_0$ where $\epsilon_0$ is a positive
number less than half the time needed for $\phi$ to move $L^{(0)}$ outside $V_0$.

Assume we have constructed $\phi_{i,t}$ and $B_i$ up to $i=k$ for some $k \geqslant 0$. Let
$U_k$ be an open neighborhood of $L^{(k)}$ such that $\phi_t=\phi_{0,t}\circ \cdots  \circ \phi_{k,t}$
on $U_k$ for $t \leqslant \epsilon_k$ and such that $L^{(k+1)}\setminus U_k$ is diffeomorphic to a
finite disjoint union of disks. Then $L^{(k+1)}\setminus U_k$ is contained in the
interior of a Darboux ball $B_{k+1}$ (each isotropic disk is, and then one can
connect the disjoint Darboux balls along transverse arcs). The section of
$TV/\xi \to V \times [0,1]$ corresponding to the isotopy 
$\phi_{k,t}^{-1}\circ \cdots  \circ \phi_{0,t}^{-1} \circ \phi_t$ vanish on $U_k\times[0,\epsilon_k]$ by
assumption. Multiply this section by a function equal to $1$ on a neighborhood
$V_{k+1}$ of $L^{(k+1)}\setminus U_k$ and supported in the interior of $B_{k+1}$ to get
a contact isotopy $\phi_{k+1,t}$. 
We have $\phi_{k+1,t}=\phi_{k,t}^{-1}\circ \cdots  \circ \phi_{0,t}^{-1}\circ \phi_t$ near $L^{(k+1)}$ and 
for $t \leqslant \epsilon_{k+1}$, where $\epsilon_{k+1}$ is positive, less than $\epsilon_k$, and less than
half the time needed for $\phi_{k+1}$ to move $L^{(k+1)}\setminus U_k$ outside $V_{k+1}$.
The result is now proved by induction.
\end{proof}

\begin{remark}\label{rem:cellcomplex}
If $(W,\omega,Z,f)$ is a Weinstein manifold of finite type such that $Z$ is Morse-Smale,
then its skeleton $L$ as well as its lift in the contactization satisfies the tameness
assumption of \cref{lem:frag}. Indeed, one may first deform the function $f$ (or $f+t^2$
in the contactization $W\times \R$ where $t$ is the $\R$-coordinate)
without changing $Z$, so that $f$ is ordered and for each $i$, the critical
points of index $i$ all lie in the same level set $\{f=i\}$. Now given any
neighborhood $V_i$ of $L^{(i)}$, we may deform $f$ without changing critical values
so that the sublevel set $U_i=\{f<i+\frac{1}{2}\}$ is contained in $V_i$.
Then $L^{(i+1)}\setminus U_i=L^{(i+1)}\cap \{i+\frac{1}{2}\leq f \leq i+1\}$
is diffeomorphic to a finite disjoint union of disks.
\end{remark}

\section{Transversality for contact transformations}
\label{sec:transversality_transfo}

In this section, we prove \cref{thm:thom-mather-contact}, a general
transversality theorem for multi-jets of families of contact diffeomorphisms.
In \cref{sec:cleaning}, it will ensure that certain properties of contact
isotopies hold after an arbitrarily small perturbation.

Let $(V, \xi)$ be a contact manifold, and let $B$ be any manifold.
We denote by $\D_B(V, \xi)$ the space of families of contact transformations of
$(V, \xi)$ parametrized by $B$, i.e. maps $f \co B \times V \to V$ such that each 
$f_b := f(b, \cdot) \co V \to V$ is a contact transformation: $(f_b)_* \xi = \xi$ for
all $b$. This space is equipped with the strong $C^\infty$ topology, which makes it a
Baire space:  a residual subset (i.e. a countable intersection of dense open
subsets) is dense. Note that $f_b$ is automatically a local diffeomorphism if
$f$ is in $\D_B(V, \xi)$, and the subset of such maps where $f_b$ is a global
diffeomorphism for all $b$ is open.

Inside the space $J^k_l(B \times V, V)$ of $k$-jets of germs of maps from $B \times V$ to
$V$ at $l$ distinct points, we consider the subspace $J^k_l(B \times V, V ; \xi)$
coming from such contact families.
Since all contact structures are locally isomorphic, we have 
$J^0_l(B \times V, V ; \xi) = J^0_l(B \times V, V)$. However for $k\geq 1$, one has a strict
inclusion. These subsets are still nice according to the following proposition,
whose proof is postponed to the end of the section.

\begin{proposition}\label{prop:Rksubmanifold}
For each $k$ and $l$, $J^k_l(B \times V, V ; \xi)$ is a submanifold of 
$J^k_l(B \times V, V)$ and the projection 
$J^{k+1}_l(B \times V, V ; \xi) \to J^k_l(B \times V, V ; \xi)$ is a submersion.
\end{proposition}

Our version of the Thom-Mather transversality theorem for families of contact
diffeomorphisms can now be stated as follows.

\begin{theorem}\label{thm:thom-mather-contact}
The families of contact diffeomorphisms whose multijet extension is transverse
to a given submanifold $\Sigma$ of $J^k_l(B \times V, V ; \xi)$ form a residual subset
of $\D_B(V, \xi)$.
\end{theorem}

\begin{remark}\label{rem:stratif}
  Since a countable intersection of residual subsets is still residual, one may
  impose simultaneously countably many transversality conditions possibly with
  varying $(k,l)$. In particular, it applies to a stratified subset. Openness
  of the set of families satisfying the transversality condition is more
  subtle, and requires additional properties of the stratification. Since we do
  not need this property in our application, we will not discuss it further,
  and content ourselves with residual subsets.
\end{remark}

\begin{proof}
Pick a contact form $\alpha$ for $\xi$. To any family $f\co B \times V \to V$
of contact diffeomorphisms, we associate its Legendrian graph as follows%
\footnote{The intrinsic definition, without using $\alpha$, would be less convenient here.}.
We have, for $(b,v) \in B \times V$, $(f^*\alpha)_{(b,v)} = \mu_v+e^{g(b,v)} \alpha$ where $\mu_v$ is a
$1$-form on $B$ smoothly depending on $v$ and $g$ is a function on $B \times V$.
With these notations we define $\Lambda_f \co B \times V \to M := T^*B \times V \times V \times \mathbb{R}$
by $\Lambda_f(b,v) = (\mu_v,v,f(b,v),g(b,v))$ and compute that it is Legendrian for
the contact form $\lambda_B + e^t\alpha_1 - \alpha_2$
where $\lambda_B$ is the canonical $1$-form on $T^* B$ and $\alpha_i$ is the pullback
of $\alpha$ under the $i$-th projection to $V$. Let $\tau$ be the projection of $M$ onto
the second $V$ factor. The same computation shows any Legendrian section $\sigma$ of
$M \to B \times V$ gives rise to a family $\tau \circ \sigma$ of  local contact diffeomorphisms.
Hence the space of families of contact diffeomorphisms now sits as an open set
in the space of Legendrian sections of $M$. The statement of the theorem
is therefore equivalent to the fact that the set of Legendrian sections of $M$
whose corresponding family has its multijet transverse to $\Sigma$ is a
residual subset. 

Let us now have a look at the Legendrian graph construction
at the level of multijets. Define $M^{(k)}_l$ to be the $(k, l)$-multijet extension of
$M \to B \times V$, i.e. the space of $k$-jets of sections at $l$ distinct points.
Let $\RelLeg^k_l \subset M^{(k)}_l$ be the differential relation corresponding
to multijets of Legendrian sections of $M$. We set $X = B \times V \times V$
seen as a bundle over the first two factors $B\times V$, we identify
$X^{(k)}_l$ with $J^k_l(B\times V,V)$ and denote $\Rel^k_l=J^k_l(B \times V, V ; \xi)$
seen as sitting in $X^{(k)}_l$. We have a natural
submersion $p^k_l:M^{(k)}_l \to X^{(k)}_l$ induced by the projection $M \to X$
which, thanks to the bijection between Legendrian sections and families of
contact transformations, satisfies $p^k_l(\RelLeg^k_l) = \Rel^k_l$. The
Legendrian graph construction also gives a map 
$\lambda^k_l:\Rel^{k+1}_l \to \RelLeg^k_l$ sending the $(k+1)$-jet of a family $f$ at some
point to the $k$-jet of $\Lambda_f$ at this point (observe that $\Lambda_f$ involves $df$,
hence the shift from $k+1$ to $k$).
Assume for a moment that $\RelLeg^k_l$ is a submanifold of $M^{(k)}_l$, a fact that
we will prove below. We know from Proposition \ref{prop:Rksubmanifold} that the projection
$\pi_l^k:\Rel^{k+1}_l \to \Rel^k_l$ is a submersion, hence the map
$p^k_l:\RelLeg^k_l \to \Rel^k_l$ is also a submersion since
$\pi_l^k = p^k_l\circ \lambda^k_l$. Therefore a family of contact diffeomorphisms has
its multijet extension transverse to $\Sigma$ if and only if its Legendrian graph 
has its multijet transverse to $(p^k_l)^{-1}(\Sigma)$.
What remains to be proved is that $\RelLeg^k_l$ is a submanifold and that the
set of Legendrian sections of $M$ whose multijet is transverse to a given
submanifold of $\RelLeg^k_l$ is a residual subset.

Let $s_0$ be a fixed Legendrian section of $M$ and pick a contact embedding
of a neighborhood of the zero-section of $J^1 (B \times V)$ to $M$
which maps the zero-section to $s_0$ over $\Id_{B \times V}$, as provided by
Weinstein's tubular neighborhood theorem. Although such a map cannot be
compatible with projections of both sides on $B \times V$ (fibers have different
contact geometry), the sections sufficiently close to $s_0$
can also be canonically viewed as sections of $J^1(B \times V)$. The key fact is then
that Legendrian sections of $J^1(B \times V)$ are exactly the holonomic ones. At the
level of multijets, this provides diffeomorphisms (near the multijet extension
$j^k_l s_0$ of $s_0$) from the multijet extensions $(J^1(B \times V))^{(k)}_l$ of 
$J^1(B \times V)$ to the multijet extensions $M^{(k)}_l$ of $M$, which maps
$J^{k+1}_l(B \times V)$ to $\RelLeg^k_l$. This shows that $\RelLeg^k_l$ is a
submanifold, since $J^{k+1}_l(B \times V)$ is a submanifold of 
$(J^1(B \times V))^{(k)}_l$. Moreover, the classical Thom-Mather theorem, applied at
order $k+1$, implies that the space of functions $B \times V \to V$ whose 
$(k+1, l)$-multijet is transverse to some submanifold of $J^{k+1}_l(B \times V)$ is a
residual subset.
Hence so is the space of Legendrian sections of $M$ whose $(k, l)$-multijet
extension is transverse to the corresponding submanifold of $\RelLeg^k_l$.
\end{proof}

It remains to prove Proposition \ref{prop:Rksubmanifold}.
For $p\in \N$, $n\in \N$ and $k\in \N$, we define $G^k_{p,2n+1}$
to be the set of $k$-jets at the origin of maps $f:\R^p \times \R^{2n+1} \to \R^{2n+1}$
such that $f(0,0)=0$ and $f_0=f(0,\cdot):\R^{2n+1} \to \R^{2n+1}$ is a local diffeomorphism.
This is a Lie group for the parameterwise composition. Note that $G^0_{p,2n+1}$ is the trivial group.
Moreover, we have projections $G^k_{p,2n+1} \to G^{k-1}_{p,2n+1}$ which are surjective Lie group homomorphisms,
and hence submersions.

Now adding the constraint that for all $b \in \R^p$, $f_b : \R^{2n+1} \to \R^{2n+1}$ is a
local \emph{contact} diffeomorphism (for the contact structure $\xi$) defines a subgroup
$H^k_{p,2n+1}$ of $G^k_{p,2n+1}$. The following lemma provides an explicit
description of $H^k_{p,2n+1}$, which implies in particular that it is a closed
subgroup\footnote{This fact is not obvious because a sequence $f_n$ of families
of contact transformations whose $k$-jet converges at a point does not
necessarily converges near that point.}, hence a submanifold by Cartan's closed
subgroup theorem.

\begin{lemma}\label{lem:description_Hk}
For $i\geq 0$, let $E_i$ be the bundle $\Lambda^i(\R^{2n+1})$ pulled back
by the projection $\R^p \times \R^{2n+1} \to \R^{2n+1}$. Each map
$f \co \R^p \times \R^{2n+1} \to \R^{2n+1}$ determines sections $\omega_2(f)$ and
$\omega_3(f)$ respectively of $E_2$ and $E_3$ by the formulas
\begin{itemize}
  \item $\omega_2(f)_{(b, v)} = (f_b^* \alpha \wedge \alpha)_v$,
  \item $\omega_3(f)_{(b, v)} = (f_b^*d\alpha \wedge \alpha - f_b^*\alpha \wedge d\alpha)_v$,
\end{itemize}
where $f_b=f(b,.): \R^{2n+1} \to \R^{2n+1}$.
The subgroup $H^k_{p,2n+1}$ consists of the $k$-jets at the origin of maps
$f$ such that $f(0,0)=0$ and the $(k-1)$-jets at the origin of the corresponding
sections $\omega_2(f)$ and $\omega_3(f)$ vanish (observe that these depend only
on the $k$-jet of $f$ at the origin).
\end{lemma}

\begin{proof}
If $f:\R^p\times \R^{2n+1} \to \R^{2n+1}$ is a family of local contact
diffeomorphisms then $\omega_2(f)$ and $\omega_3(f)$ vanish identically.

Conversely if the $k$-jet of $f$ is such that $\omega_2(f)$
and $\omega_3(f)$ vanish at order $k-1$, we will prove that we can turn $f$
into a family of local contact diffeomorphisms without changing its $k$-jet
at the origin. For this we follow the path method.

In the following, the differential forms always depend on the parameter $b$,
i.e. are seen as sections of $E^i$, though we often drop the subscript $b$
for notational convenience. We set $\alpha_1=f_b^*\alpha$, $\alpha_t = (1 - t)\alpha + t\alpha_1$
and $\xi_t = \ker \alpha_t$. Note that $\xi_t$ is a contact structure near
the origin for all $t$ and all $b$. Indeed, \emph{at the origin}, we have
$\alpha_1 \wedge \alpha = 0$, hence $\ker \alpha_t = \xi$, and then $d(\alpha_1 \wedge \alpha) = 0$ implies that 
$d\alpha_t\rst{\xi}$ is a multiple of $d\alpha\rst{\xi}$, hence symplectic.
The vanishing condition on $\omega_2(f)$ and $\omega_3(f)$ imply, after
differentiating with respect to $t$,
\begin{align}
&\dot{\alpha_t} \wedge \alpha_t =o(|b,v|^{k-1}) \label{eq:1}\\
&\dot{\alpha_t}\wedge \d \alpha_t - \alpha_t \wedge \d \dot{\alpha_t} =o(|b,v|)^{k-1}. \label{eq:2}
\end{align}

We will construct a local isotopy $\Phi_t$ (fibered over $\Id_{\mathbb{R}^p}$)
such that, denoting by $\phi_t$ the restriction of $\Phi_t$ to some unspecified slice
$\{b\} \times \mathbb{R}^{2n+1}$, $\phi_t^*\xi_t = \xi_0$. It will be generated by a vector field
$X_t$ that we decompose as $X_t = f_t R_t + Y_t$ where $R_t$ is the Reeb vector
field of $\alpha_t$ and $\alpha_t(Y_t) = 0$. The usual discussion of the path method in this
context (see e.g. \cite[Page~60]{Geiges_book}) ensures that $\phi_t$ will pull back
$\xi_t$ onto $\xi_0$ as soon as $(df_t + \dot \alpha_t + Y_t \intprod d\alpha_t) \wedge \alpha_t = 0$. This is
equivalent to $(Y_t \intprod d\alpha_t)\rst{\xi_t} = -(df_t + \dot \alpha_t)\rst{\xi_t}$ and, since
$d\alpha_t\rst{\xi_t}$ is non-degenerate, this uniquely defines $Y_t$. What is specific
to our situation is that we need to ensure that $f_t$ and $Y_t$ are both
$o(|b,v|^k)$, so that $\Phi_t(b,v) =(b, v + o(|b,v|^k)$. Per the above discussion,
the estimate on $Y_t$ is equivalent to $(df_t + \dot \alpha_t) \wedge \alpha_t = o(|b,v|^k)$.

We set $\gamma_t = \dot \alpha_t - \dot \alpha_t(R_t)\alpha_t$. Plugging $R_t$ into \eqref{eq:1}
gives $\gamma_t = o(|b,v|^{k-1})$. In addition, \eqref{eq:2} ensures that
$d\gamma_t \wedge \alpha_t = \gamma_t \wedge d\alpha_t + o(|b,v|^{k-1})$, hence $d\gamma_t \wedge \alpha_t = o(|b,v|^{k-1})$.
Recall the de Rham homotopy formula, for any $h \co M \times [0, 1] \to M$,
$h_1^* - h_0^* = H \circ d + d \circ H$ where $H\eta = \int_0^1 \partial_s \intprod h^*\eta$.

Darboux's theorem, with parameters, ensures the existence of coordinates,
smoothly varying in $t$ (and $b$), such that $\alpha_t = dz + \lambda$ where $\lambda$ is the radial
Liouville form on $\mathbb{R}^{2n}$: $\lambda = \Sigma(x_idy_i - y_idx_i)/2$. In these
coordinates, we will use the Liouville homotopy 
$h \co (x, y, z, s) \mapsto (sx, sy, z)$.  The corresponding operator $H$ on
differential form satisfies $Ho(|b,v|^j) = o(|b,v|^{j+1})$ for every $j$,
because $dh/ds = O(|v|)$. So we can set $f_t = -H\gamma_t$, and have 
$f_t = o(|b,v|^k)$. 
In addition, since $\gamma_t(\partial_z) = \gamma_t(R_t) = 0$, we have $h_0^*\gamma_t = 0$.
Hence the homotopy formula gives $df_t = -\gamma_t + Hd\gamma_t$. In the following
computation, we will use this, the observation 
$h^*\alpha_t = \alpha_t + (s^2 - 1)\lambda$, and its consequence $\partial_s \intprod h^*\alpha_t = 0$.
\begin{align*}
 (df_t + \dot \alpha_t) \wedge \alpha_t &= Hd\gamma_t \wedge \alpha_t \\
  &= \int (\partial_s \intprod h^*d\gamma_t) \wedge \alpha_t \\
  &= \int (\partial_s \intprod h^*d\gamma_t) \wedge (h^*\alpha_t - (s^2 - 1)\lambda) \\
  &= \int \partial_s \intprod h^*(d\gamma_t \wedge \alpha_t) + \int (1 - s^2)(\partial_s \intprod h^*d\gamma_t) \wedge \lambda.
\end{align*}
Since $d\gamma_t \wedge \alpha_t = o(|b,v|^{k-1})$, the first term is $o(|b,v|^k)$.
In the second term, we have $\gamma_t = o(|b,v|^{k-1})$, hence $d\gamma_t =
o(|b,v|^{k-2})$, and then $\partial_s \intprod h^*d\gamma_t = o(|b,v|^{k-1})$.
Since $\lambda = O(|v|)$, everything is $o(|b,v|^k)$ as needed.
\end{proof}

\begin{proof}[Proof of Proposition \ref{prop:Rksubmanifold}]
We use the same notations as in the proof of \cref{thm:thom-mather-contact}.
We want to prove that each $\Rel^k_l$ is a submanifold of
$X^{(k)}_l$ and $\Rel^{k+1}_l \to \Rel^k_l$ is a submersion.
It suffices to prove it for $l = 1$, since for $l\geq 2$, $\Rel^k_l$ is an
open set in the $l$-fold product of $\Rel^k$.

Let $p=\dim B$ and $2n+1=\dim V$. From Lemma \ref{lem:description_Hk}, we see
that $H^k_{p,2n+1}$ is a closed subgroup of $G^k_{p,2n+1}$, and hence a Lie
subgroup according to É. Cartan. The Lie group homomorphism
$H^k_{p,2n+1} \to H^{k-1}_{p,2n+1}$ is surjective and is thus
a submersion.

To complete the proof, we only need to find local trivializations
$X^{(k)} \simeq \R^p \times \R^{2n+1} \times \R^{2n+1} \times G^k_{p,2n+1}$
which maps $\Rel^k$ to $\R^p \times \R^{2n+1} \times \R^{2n+1} \times H^k_{p,2n+1}$
and commute with the projections $X^{(k)} \to X^{(k-1)}$ and
$\R^p \times \R^{2n+1} \times \R^{2n+1} \times H^k_{p,2n+1} \to
\R^p \times \R^{2n+1} \times \R^{2n+1} \times H^{k-1}_{p,2n+1}$.

For this it will be convenient to use the Heisenberg group structure
to choose, for each $y$ in $\mathbb{R}^{2n + 1}$, a contactomorphism depending smoothly
on $y$ and sending the origin to $y$. Indeed consider
$\lambda=1/2\sum (p_i\d q_i - q_i \d p_i)$, $\alpha=\d z + \lambda$ the
standard contact form on $\R^{2n+1}$, and $\xi=\ker \alpha$. To this we
associate the Heisenberg Lie group structure on $\R^{2n+1}$ where 
$(p_1, q_1, z_1)\cdot(p_2, q_2, z_2) = (p_1 + p_2, q_1 + q_2, z_1 + z_2 + d\lambda((p_1, q_1), (p_2, q_2))$. 
Then $\xi$ is invariant under right translation 
$R_y : x \mapsto x\cdot y$.

We now build the trivializations. Fix a point $(b,v,w) \in B\times V \times V$, 
pick local charts $\psi :(U_b,b) \simeq (\R^p,0)$, Darboux charts 
$\phi:(U_v,v) \simeq (\R^{2n+1},0)$ and $\theta:(U_w,w) \simeq (\R^{2n+1},0)$ (having
images of charts that are whole spaces is convenient, and easily arranged since
an open ball in standard contact space is isomorphic to the whole space). 
The k-jet at $(b',v') \in U_b \times U_v$ of a map $f:B\times V \to V$ with 
$f(b',v')=w'$ in $U_w$ is sent to $(\psi(b'),\phi(v'),\theta(w'), j^k g(0,0))$
where $g:\mathbb{R}^p \times \mathbb{R}^{2n+1} \to \mathbb{R}^{2n+1}$ is given by
\[
  g(x,y) = (R_{\theta(w')}^{-1} \circ \theta \circ f)
  \left(\psi^{-1}\left(x+\psi(b')\right), \phi^{-1}\circ R_{\phi(v')}(y)\right).
\]
Which is indeed a family of contact diffeomorphisms sending $(0, 0)$ to $0$.
This map to $\R^p \times \R^{2n+1} \times \R^{2n+1} \times H^k_{p,2n+1}$
is indeed a diffeomorphism: the inverse map sends
$(x, y, z, j^kg(0, 0))$ to $j^kf(b', v')$ where $b' = \psi^{-1}(x)$, $v' = \phi^{-1}(y)$,
and 
\[
  f(b'', v'') = (\theta^{-1} \circ R_z \circ g)\left(\psi(b'')- x, \phi(v'')\cdot y^{-1}\right). \qedhere
\]
\end{proof}

\section{Cleaning contact isotopies}
\label{sec:cleaning}

The proof of the decomposition theorem in \cref{sec:decomposition} will be
reduced to contact isotopies satisfying technical hypotheses. The goal of this
section is to prove that these hypotheses can be ensured by perturbation.

\begin{definition}\label{def:clean}
Let $(V,\xi)$ be a contact manifold, $L_-$ and $L_+$ be isotropic submanifolds of $V$.
A contact isotopy $f:I \times V\to V$ is $(L_-, L_+)$-clean if its restriction
$g:I \times L_-\to V$ satisfies :
\begin{enumerate}[label=(C-\arabic*), ref=(C-\arabic*)]
    \item\label{it:transLp}
      $g$ is transverse to $L_+$,
    \item\label{it:uneCollision}
      $\forall (t,t',x)\in I\times I\times L_-$, if $g(t,x) \in L_+$ and
      $g(t',x) \in L_+$ then $t = t'$,
    \item\label{it:immersion}
      $\forall (t,t',x) \in I\times I\times L_-$, if $g(t,x) \in L_+$, then $g$ is
      an immersion at $(t',x)$,
    \item\label{it:injection}
      $\forall (t,t',t'',x,x') \in I\times I\times I \times  L_-\times L_-$,
      if $g(t,x) \in L_+$ and $g(t',x')=g(t'',x)$ then $t'=t''$ and $x'=x$,
    \item\label{it:almostTransverseXi}
      $\forall (t,x)\in I\times L_-$, if $g(t,x) \in L_+$ then for any
      equation $\alpha$ of $\xi$ the function $\phi(s) = \alpha(\frac{\del g}{\del s}(s,x))$
      vanishes transversely, i.e. $\phi(s) = 0 \Rightarrow \phi'(s)\neq 0$.
  \end{enumerate}
If $L_-$ and $L_+$ are isotropic complexes (see definition \ref{def:isotropiccomplex})
rather than submanifolds, we say that $f$ is $(L_-,L_+)$-clean if $f$ is
$(L_-^{(i)} \setminus L_-^{(i-1)},L_+^{(j)}\setminus L_+^{(j-1)})$-clean
for all $i,j \in \{0,\dots,n\}$.
\end{definition}

Note that if $L_-$ is subcritical, that is $L_- = L_-^{(n-1)}$, the conditions above
simply reduce to $f_t(L_-)\cap L_+ = \emptyset$ for all $t$.

\begin{proposition}\label{prop:cleaning}
  Let $(V,\xi)$ be a contact manifold of dimension 
$2n+1$ with $2n+1 \geq 5$, and $L_-$, $L_+$ disjoint isotropic complexes.
The set of $(L_-, L_+)$-clean isotopies is residual in $\D_I(V, \xi)$.
\end{proposition}
\begin{proof}
  For short, we will write $\D$ for the space $\D_I(V, \xi)$ of contact isotopies.
  Because of \cref{rem:stratif}, it is enough to prove the result
  in the case where $L_-$ and $L_+$ are submanifolds (of constant dimension).
  In all codimensions computations below, we
  will assume that $L_-$ and $L_+$ are of dimension $n$. If not then all
  codimensions would be higher, and the conclusion even stronger than needed.
  All the relevant submanifolds in jet spaces will be defined by independent
  conditions whose codimensions will thus add up to the submanifold codimension.
  A condition asking for some $x \in X$ to be in a submanifold $X'$ has
  codimension $\codim(X')$. A condition asking that $x = x'$ in $X$ means that
  $(x, x')$ is in the diagonal $\Delta_X \subset X \times X$ hence has codimension
  $\codim(\Delta_X) = \dim X$.

  \Cref{it:transLp} for $f \in \D$ is implied by (in fact equivalent to) the
  transversality of $j^0f$ to
  \[
    \Sigma_1 = I \times L_- \times L_+.
  \]
  Indeed, assume $j^0f$ is transverse to $\Sigma_1$. Let $\pi$ be the projection
  of $TJ^0(I \times V, V)$ onto the normal bundle $\nu\Sigma_1 = \{0\} \times \nu L_- \times \nu L_+$.
  At any $(t, x)$ such that $j^0f(t, x)$ is in $\Sigma$,
  $\pi \circ T_{(t, x)}f$ is surjective. In particular it is surjective onto
  $\{0\} \times \{0\} \times \nu L_+$. Hence $T_{(t, x)}f\rst{T_tI \times T_x L_-}$ is surjective
  onto $\nu L_+$.

  \Cref{it:uneCollision} is equivalent to asking that $j^0_2f$ avoids:
  \[\begin{split}
      \Sigma_2 = \{ & ((t,x,y),(t',x',y')) \in (I \times V \times V)^2, \\
                    & x \in L_-,  y \in L_+, x' = x, y' \in L_+ \}.
  \end{split}\]
  The codimension of $\Sigma_2$ in $J^0_2(I \times V, V)$ is
  $\codim(L_-) + \codim(L_+) + \dim(V) + \codim(L_+) = 5n + 4$.
  For $n\geq 1$, $5n+4 > 4n + 4 = \dim((I \times V)^2)$, so $j_2^0 f$ avoiding
  $\Sigma_2$ is equivalent to $j^0_2f \pitchfork \Sigma_2$.

  In order to discuss conditions involving $j^1f$, we will identify
  $\Hom(T_tI, T_yV)$ with $T_yV$ by $\varphi \mapsto \varphi(\partial_t)$, so that $T_{(t, x)}f$ will be
  represented by some $(A, b)$ in $\Hom(T_xV, T_{f(t, x)}V) \times T_{f(t, x)}V$.

  Observe that $g$ being an immersion at $(t',x)$ is equivalent to
  $T_{t', x}g(\partial_t) \not \in T_{g(t, x)} g_t(L_-)$. Hence condition \ref{it:immersion}
  is equivalent to asking that $j^1_2f$ avoids
  \[\begin{split}
      \Sigma_3 = \{ & ((t,x,y,A,b),(t',x',y',A',b')), \\
                    & x \in L_-, y \in L_+, x' = x, b'\in A(T_{x'} L_-)\}.
  \end{split}\]
  This $\Sigma_3$ is a submanifold of $J^1_2(I \times V, V ; \xi)$ whose codimension is
  $\codim(L_-) + \codim(L_+) + \dim(V) + \codim(L_-) = 5n+4$.
  For $n\geq 1$, $5n+4 > 4n+4 = \dim((I \times V)^2)$, so $j^1_2f$ avoiding $\Sigma_3$ is
  equivalent to $j^1_2f \pitchfork \Sigma_3$.

  \Cref{it:injection} is equivalent to $j^0_3f$ avoiding
  \[\begin{split}
      \Sigma_4 = \{ & ((t,x,y),(t',x',y'),(t'',x'',y'')), \\
                     & x \in L_-, y \in L_+, x' = x, x'' \in L_-, y' = y''\}.
  \end{split}\]
  This $\Sigma_4$ is a submanifold of $J_3^0(I \times V, V)$ of codimension
  $\codim(L_-) + \codim(L_+) + \dim(V) + \codim(L_-) + \dim(V) = 7n + 5$.
  For $n \geq 2$, $7n+5 > 6n+6 = \dim((V \times I)^3)$ so $j^0_3f$ avoiding $\Sigma_4$ is
  equivalent to $j^0_3f \pitchfork \Sigma_4$.

  \Cref{it:almostTransverseXi} asks that if $g(t,x) \in L_+$ for some
  $(t, x)$ then the function $t \mapsto \alpha(T_{(t, x)}g(\partial_t))$ vanishes
  transversely.
  Consider the following submanifold of $J^1_2(I \times V, V ; \xi)$:
  \[\begin{split}
      \Sigma_5 = \{& ((t,x,y,A,b),(t',x',y',A',b')), \\
                   & x\in L_-, y \in L_+, x' = x, \alpha(b') = 0 \}.
  \end{split}\]
  We claim that transversality of $j^1_2f$ to $\Sigma_5$ implies condition
  \ref{it:almostTransverseXi} (it is actually equivalent
  under condition \ref{it:transLp} but we will not need this equivalence).
  Assume $j^1_2f$ is transverse to $\Sigma_5$ for some $f$ in $\D$. Let
  $((t,x),(t',x))$ be a point in $(I \times V)^2 \setminus \Delta$  sent to $\Sigma_5$
  by $j^1_2f$. The normal space of $\{x = x', x \in L_-\}$ in $V \times V$ at $(x, x)$ is
  isomorphic to $\nu_xL_- \times T_xV$ via $[(u, v)] \mapsto ([u], u - v)$. Also, choosing a
  contact form $\alpha$ allows to identify $TV/\xi$ with $V \times \mathbb{R}$, and we set
  $\phi_f(t, x) = \alpha(T_{(t, x)}f(\partial_t))$. Using these identifications and
  notations, transversality of $j^1_2f$ at $p$ becomes surjectivity of the map
    \begin{gather*}
    T_tI \times T_xV \times T_{t'}I \times T_xV \to \nu_xL_- \times \nu_yL_+ \times T_xV  \times T_0\mathbb{R}, \\
    (\tau, u, \tau', u') \mapsto
    \Big([u],\; [T_{(t, x)} f(\tau, u)],\; u - u',\; T_{(t', x')}\phi_f(\tau', u')\Big).
  \end{gather*}
  In particular, this map is surjective onto $\{0\} \times  \nu_yL_+ \times \{0\} \times \{0\}$,
  so $T_{(t, x)}f$ induces an isomorphism from $T_tI \times T_xL_-$ to
  $\nu_{f(t, x)}L_+$. Next we use surjectivity onto $\{0\} \times \{0\} \times \{0\} \times T_0\mathbb{R}$
  to get existence of $(\tau, u, \tau', u')$ such that
  $u$ is in $T_xL_-$, $T_{(t, x)} f(\tau, u)$ is in $T_{f(t, x)}L_+$, $u' = u$, and
  $T_{(t', x')}\phi_f(\tau', u')$ is not zero. The first two conditions imply $u = 0$
  by the isomorphism above. Hence $u' = 0$ by the third condition and the last
  condition becomes $\partial_t\phi_f(t', x') \neq 0$ as desired.

  Transversality to $\Sigma_i$ for all $i\in \{1, \dots, 5\}$ determines
  a residual subset of $\D$ by the transversality theorem,
  \cref{thm:thom-mather-contact}, and we have seen that it implies cleanness.
\end{proof}

\section{The decomposition theorem}
\label{sec:decomposition}

The goal of this section is to prove the next result, which implies
\cref{thm:decomp_intro} from the introduction, and is our main geometrical
ingredient.

\begin{theorem}\label{thm:decomposition}
  Let $(V, \xi)$ be a contact manifold and $L_-$, $L_+$ be two disjoint
  compact isotropic complexes such that $L_-$ is loose in the complement of $L_+$.
  Any contact isotopy is homotopic, with fixed end-points, to a composition
  $f_t = g_t \circ f^-_t \circ f^+_t \circ g'_t$, where $g_t$ and $g'_t$ have support in
  Darboux balls $B$ and $B'$, and $f^\pm$ has compact
  support away from $\varphi_1(L_\pm)$ for some contact isotopy $\varphi$.
\end{theorem}

\begin{remark}
  Carefully reading the proof of the above \lcnamecref{thm:decomposition}
  reveals that $\varphi$ can be chosen arbitrarily small in $C^1$ topology, and
  with support in an arbitrarily small neighborhood of $L_- \cup L_+$. Since we have
  no use for this refinement, we neither include it in the statement, nor
  follow the size of $\varphi$ along various reductions.
\end{remark}

The discussion of \cref{thm:decomposition} will use the following technical
definition.

\begin{definition}
  Let $L_-$ and $L_+$ be two subsets of a closed contact manifold $(V, \xi)$.
  An $(L_-, L_+)$-decomposition for a contact isotopy $f$ is a factorization for
  all $t$
  \[
    f_t = g_t \circ f^-_t \circ f^+_t \circ g'_t
  \]
  where $g$ and $g'$ have compact support in Darboux balls $B$ and $B'$,
  and each $f^\pm$ has compact support outside of $L_\pm$.
\end{definition}

Using the above definition, the conclusion of \cref{thm:decomposition}
is that $f$ is homotopic to an isotopy admitting an
$(\varphi_1(L_-), \varphi_1(L_+))$-decomposition for some contact isotopy $\varphi$.

The proof of the decomposition \lcnamecref{thm:decomposition} splits into two
independent parts. First we explain in \cref{sec:reduction_decomp} that the
approximation result, \cref{prop:cleaning}, can be used to reduce to clean
isotopies, as introduced in \cref{def:clean}. Then the crucial part, where the
flexibility assumption appears, is \cref{prop:main_decomp} below, which 
will be proved in \cref{sub:proof_decomposition}.

\begin{proposition}\label{prop:main_decomp}
  Let $(V, \xi)$ be a contact manifold of dimension $2n+1$ with $2n+1 \geq 5$,
  and $L_-$, $L_+$ two disjoint compact isotropic complexes such that $L_-$
  is loose in $V\setminus L_+$. If a contact isotopy is $(L_-, L_+)$-clean
  then it has an $(L_-, L_+)$-decomposition.
\end{proposition}

\subsection{Reduction to the clean case}\label{sec:reduction_decomp}

The goal of this section is to prove that \cref{thm:decomposition} follows from
\cref{prop:main_decomp}.
The first crucial \lcnamecref{lem:disjoint} is a simple consequence of how
contact Hamiltonian allow to cut-off contact isotopies.

\begin{lemma}\label{lem:disjoint}
  Let $K_-$ and $K_+$ be two compact subsets in a contact manifold $(V, \xi)$. Let
  $f$ be a contact isotopy of $(V, \xi)$. If, for all $t$, $f_t(K_-)$ is
  disjoint from $K_+$ then one can decompose $f_t$ as $f^-_t \circ f^+_t$ where each
  $f^\pm_t$ has support away from $K_\pm$. Alternatively, one can decompose
  $f_t$ as $f^+_t \circ f^-_t$, with the same support constraints.
\end{lemma}

\begin{proof}
  By assumption, $f([0, 1] \times K_-)$ and $K_+$ are disjoint compact subsets in $V$.
  Hence there exists a cut-off function $\rho$ with compact support which equals
	one on a neighborhood of $f([0, 1] \times K_-)$ and vanishes on a neighborhood of
  $K_+$.

  Let $X_t$ be the time-dependent vector field generating $f$ and let $H_t$ be
  the Hamiltonian function corresponding to $X_t$ (either using an auxiliary
  contact form or seeing $H_t$ as a section of $TV/\xi$).
  Let $Y_t$ be the time-dependent contact vector field corresponding to the
  Hamiltonian $\rho H_t$. Because $Y_t$ vanishes outside the support of $\rho$, its
  flow $f^+_t$ is defined for all time $t$ in $[0, 1]$ and $f^+_t$ is the identity
  on a neighborhood of $K_+$.

  In addition, $f^+_t = f_t$ on $K_-$. We set $f^-_t = (f^+_t)^{-1} \circ f_t$ or
  $f^-_t = f_t \circ (f^+_t)^{-1}$, depending on the desired decomposition order.
\end{proof}

\begin{proof}[Proof of \cref{thm:decomposition}]
  Let $f$ be any contact isotopy of $(V, \xi)$.
  Then \cref{prop:cleaning} gives a contact isotopy $\bar f$ which
  is an arbitrarily small perturbation of $f$ and is $(L_-, L_+)$-clean. 
  We choose it small enough to make sure that, for all $t$, 
  $\bar f_t^{-1} \circ f_t(L_-)$ is disjoint from $L_+$.
  Then \cref{lem:disjoint} constructs contact isotopies $\delta^-$ and $\delta^+$, with support
  disjoint from $L_-$ and $L_+$ respectively, such that, for all $t$,
  \[
    \bar f_t^{-1} \circ f_t = \delta^+_t \circ \delta^-_t.
  \]

  \Cref{prop:main_decomp} gives a $(L_-, L_+)$-decomposition of $\bar f$:
  $\bar f_t = \bar g_t \circ \bar f^-_t \circ \bar f^+_t \circ \bar g'_t$, where $g$, (resp.
  $g'$) has support in some Darboux ball $B$ (resp. $\bar B'$), and $\bar f^{\pm}$
  has support away from $L_\pm$.
  This can be rewritten as:
  \[
    f_t = \bar g_t \circ \underbrace{\bar f^-_t \circ \delta^-_t}_{=: f^-_t} \circ
    \underbrace{(\delta^-_t)^{-1} \circ \big(\bar f^+_t \circ \delta^+_t \big) \circ \delta^-_t}_{=\; \conj{(\delta^-_t)^{-1}}{\bar f^+_t \circ \delta^+_t}}
    \circ \underbrace{(\delta^+_t \circ \delta^-_t)^{-1} \circ \bar g'_t \circ (\delta^+_t \circ \delta^-_t)}_{=\; \conj{(\delta^+_t \circ \delta^-_t)^{-1}}{\bar g'_t}}.
  \]
  \Cref{lem:conj_homotopy} ensures this isotopy is homotopic to:
  \[
    t \mapsto \bar g_t \circ f^-_t \circ
    \underbrace{\conj{(\delta^-_1)^{-1}}{\bar f^+_t \circ \delta^+_t}}_{=: f^+_t} \circ
    \underbrace{\conj{(\delta^+_1\circ\delta^-_1)^{-1}}{\bar g'_t}}_{=: g'_t}
  \]
  where $f^+_t$ is relative to $(\delta^-_1)^{-1}(L_+)$, and $g'_t$ has support in 
  the Darboux ball $B' := (\delta^+_1\circ\delta^-_1)^{-1}\bar B'$, so we set $\varphi_t = (\delta^-_t)^{-1}$
  (note that $\varphi_t(L_-) = L_-$ for all $t$).
\end{proof}

\subsection{Decomposition of clean isotopies}
\label{sub:proof_decomposition}

In this section, we prove \cref{prop:main_decomp}. Remember clean isotopies were
introduced in \cref{def:clean}. Numbered conditions like \cref{it:transLp} in
the proof below refer to items in this definition.

The proof is organized into a sequence of steps. Each step uses the statements
and notations of previous steps but not their proofs, so step proofs can be
checked independently.

Let $f\co [0, 1] \times V \to V$ be an $(L_-,L_+)$-clean contact isotopy. 
The first step sets the stage without modifying $f$, essentially unpacking
consequences of the definition of clean isotopies, but also using an engulfing
argument relying on the $h$-principle for transverse arcs. 
The second step composes $f$ with Darboux ball-supported isotopies on both
sides to get $f'$ with convenient fixed loose charts, using Murphy's
flexibility theorem.
The third step uses these loose charts, and five invocations of Murphy's
theorem, to deform the isotropic isotopy $f'\rst{L_-}$ until there is no more
collision with $L_+$.
The fourth steps lifts this deformation to a deformation $f''$ of $f'$ by
post-composition with a Darboux ball-supported isotopy. 
The conclusion applies \cref{lem:disjoint} to $f''$.

In this proof, the word ball, without adjective, always mean a closed
codimension 0 ball in $V$ with smooth boundary. The word disk will always mean 
a closed codimension 0 ball in $L_-$ with smooth boundary. 

\begin{step}(See \cref{fig:step1})
\label{step:setup}
\begin{stepitems}
\item 
  There exists a finite collection of distinct points 
  $x_i \in L_-^{(n)}\setminus L_-^{(n-1)}$ and times $t_i \in (0, 1]$ such that
  $f^{-1}(L_+) \cap ([0, 1] \times L_-) = \{(t_i, x_i)\}$.

\item\label{it:emb} 
  There exists a collection of pairwise disjoint disks
  $D_i \subset L_-^{(n)}\setminus L_-^{(n-1)}$ centered at $x_i$, whose union
  is denoted by $D$, such that $f\co [0, 1] \times D \to V$ is an embedding,

\item\label{it:tub} There exists a collection of pairwise disjoint balls
  $C_i \subset V$, whose union is denoted by $C$, such that,
  $f^{-1}(C_i)\cap ([0, 1] \times L_-) = [0, 1] \times D_i$, and each
  $f_t:D_i\to C_i$ is a neat embedding for all $t$.

\item\label{it:tubdarboux} There is a Darboux ball $B$ containing $C$ in its
  interior.
\end{stepitems}
\end{step}
\begin{figure}[ht]
  \centering
  \includegraphics[width=5cm]{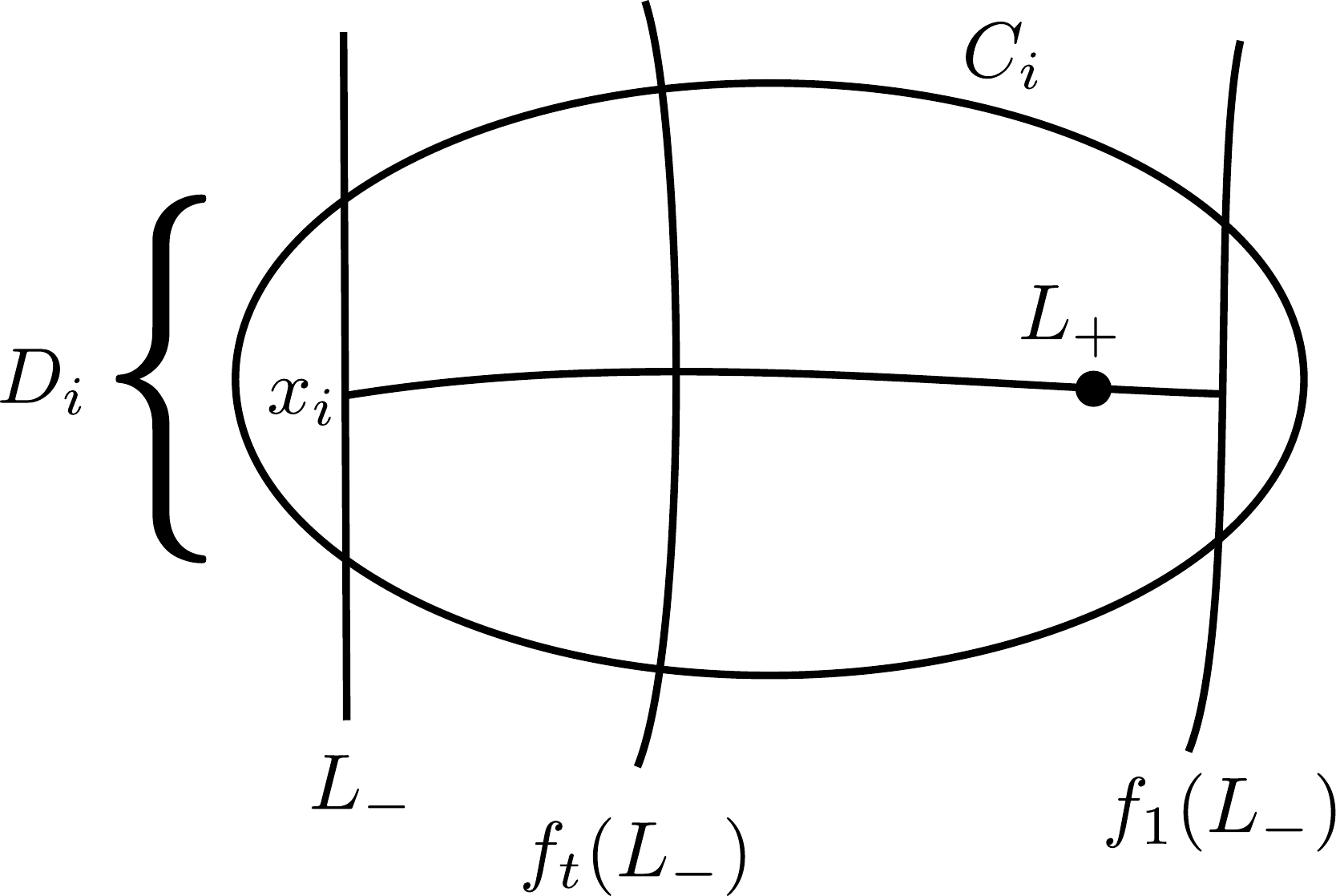}
  \caption{}
  \label{fig:step1}
\end{figure}
The set $X=f^{-1}(L_+)\cap ([0, 1] \times L_-)$ is closed in $[0, 1] \times L_-$
(since $L_+$ is closed) and hence compact (since $L_-$ is compact).
\Cref{it:transLp} in \cref{def:clean} implies that
$X=f^{-1}(L_+) \cap ([0, 1] \times (L_-^{(n)}\setminus L^{(n-1)}_-))$ and that
$X$ is discrete. Hence $X$ is a finite collection of points $(t_i, x_i)$.
\Cref{it:uneCollision} implies that the $x_i$ are pairwise distinct.
\Cref{it:injection} implies that the arcs $\gamma_i := f([0, 1] \times \{x_i\})$ are
pairwise disjoint and simple. \Cref{it:immersion} implies they are embedded.
\Cref{it:almostTransverseXi} implies that they are transverse to $\xi$ except
at finitely many points. We claim that any neighborhood $U$ of an arc $\gamma$
with this property contains a Darboux ball containing $\gamma$ in its interior. This
then allows us to construct a collection of pairwise disjoint Darboux balls
$B_i$ containing $\gamma_i$ in its interior and the required Darboux ball $B$
by a connect sum operation along transverse arcs joining north and south poles
of the balls $B_i$.
To prove the claim, pick an open ball $W$ containing $\gamma$ in its interior
and contained in $U$ and let $\{p_j\}$ be the set of points $x$ where either
$\gamma$ is not transverse to $\xi$ or $x$ is an end-point of $\gamma$. Pick
disjoint Darboux balls $A_j$ centered at the points $p_j$ and contained in $W$
and connect the balls $A_j$ along transverse arcs disjoint from $\gamma$
to get a single Darboux ball $A$ contained in $W$. Then $\gamma\cap(W\setminus A)$
consists of finitely many transverse arcs with boundary on $\partial A$, which can be pushed
inside of $A$ by a smooth isotopy since $\pi_1(W,A)=0$. By $h$-principle for transverse
arcs, there is a transverse isotopy of $\gamma$ in $W$ which takes $\gamma$ in
the interior of $A$. Lift this isotopy to a contact isotopy $\theta_t$
for $t \in [0,1]$ supported in $W$, and the ball $\theta_1^{-1}(A)$ contains the initial arc
$\gamma$ in its interior and is contained in $W$ and hence in $U$.

We now construct the disks $D_i$ and balls $C_i$.
Note first that $L_-^{(n)}\setminus L_-^{(n-1)}=L_-$ near $x_i$ according
to \cref{def:isotropiccomplex}. According to \cref{it:injection},
we have $f^{-1}(\gamma_i) \cap ([0, 1] \times L_-) = [0, 1] \times \{x_i\}$.
Moreover, $f$ immerses $[0, 1] \times L_-$ near $(t, x_i)$ for all $t$.
A compactness argument then shows that \ref{it:emb}
holds and $f^{-1}(f([0, 1] \times D_i)) = [0, 1] \times D_i$
as soon as the radius of $D_i$ is sufficiently small.
We may then extend $f([0, 1] \times D_i)$ to a ball $C_i$
such that $[0, 1] \times D_i \cap f^{-1}(\del C_i) = [0, 1] \times \partial D_i$
and $f([0, 1] \times D_i)$ is transverse to $\del C_i$.
\cref{it:tub} then holds as soon as $C_i$ is sufficiently thin.
Since $\gamma_i \subset \Int B$, we may assume that $C_i \subset \Int B$
in the above construction.

\begin{step}(See \cref{fig:step2_1,fig:step2_2})
  There exist a contact isotopy $g'$ supported in a Darboux ball $B'$,
  a contact isotopy $h$ supported in $B$, a collection of closed disks
$D''_i \subset \Int D_i \setminus \{x_i\}$ and $\epsilon \in (0, \min_i t_i)$
  such that the contact isotopy $f'_t := h_t^{-1} \circ f_t \circ (g'_t)^{-1}$ satisfies:
  \begin{stepitems}
    \item\label{it:same_coll}
      $(f')^{-1}(L_+) \cap ([0, 1] \times L_-) = \bigcup_i\{(t_i, x_i)\}$,
    \item\label{it:in_D_i''}
      $f_t=f'_t$ on $L_-\setminus D_i''$ for $t\geq \epsilon$ and
      $f'([\epsilon, 1] \times D_i'') \cap f'([\epsilon, 1] \times (L_-\setminus D_i''))=\emptyset$
    \item\label{it:same_disks}
      $f'([\epsilon, 1] \times D_i'')\subset \Int C_i$,
    \item\label{it:fixed_chart}
  for $t\geq \epsilon$, the diffeomorphisms $f'_t$ restrict to neat Legendrian
  embeddings $k_t:D_i \to C_i$ admitting a fixed loose chart
  $U_i \subset (\Int C_i \setminus (\gamma_i \cup L_+))$
  (fixed means that $k_t^{-1}(U_i)$ and $k_t(k_t^{-1}(U_i))$ are independent of $t$).
  \end{stepitems}
\end{step}

\begin{figure}[ht]
  \centering
  \begin{subfigure}[b]{5cm}
    \includegraphics[width=5cm]{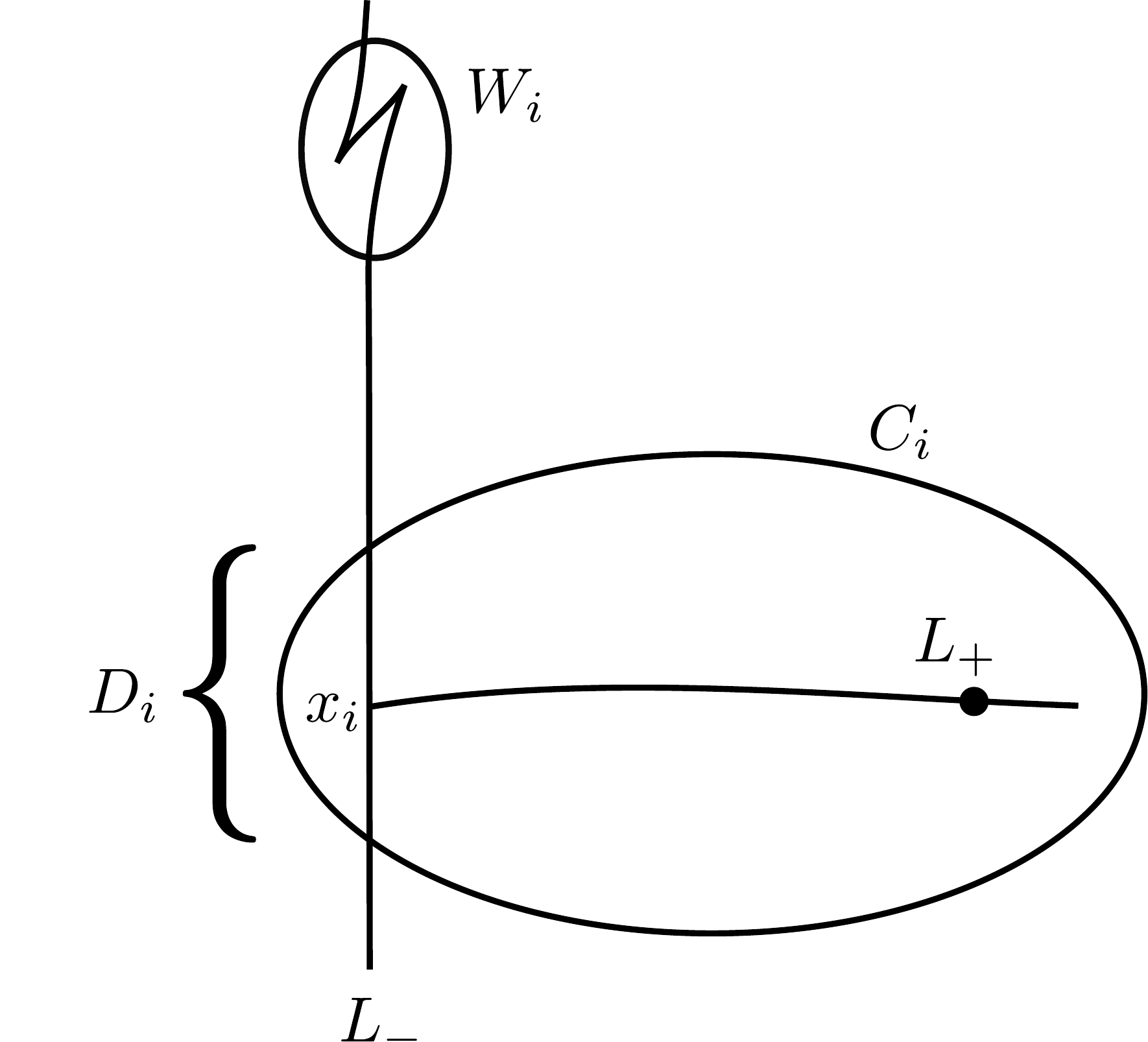}
    \caption{}
    \label{fig:step2_1_1}
  \end{subfigure}
  \begin{subfigure}[b]{5cm}
    \includegraphics[width=5cm]{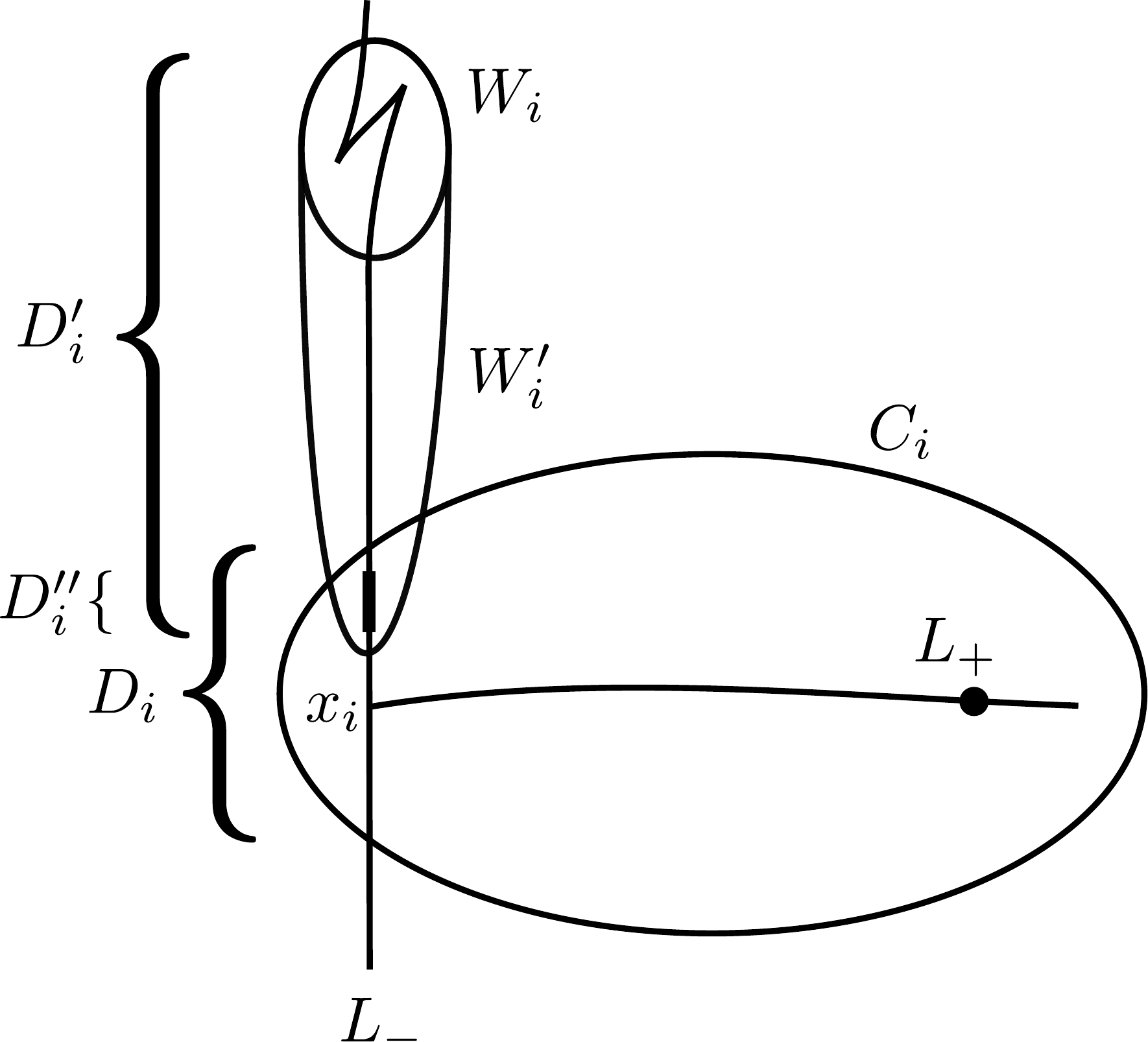}
    \caption{}
    \label{fig:step2_1_2}
  \end{subfigure}
  \caption{}
  \label{fig:step2_1}
\end{figure}
We pick a collection of disjoint loose charts $W_i$ in
$V\setminus (L_+\cup L_-^{(n-1)})$ for the connected components of
$L_-^{(n)}\setminus L_-^{(n-1)}$ containing $x_i$ as granted by \cref{def:loosecomplex}
(the fact that the loose charts can be assumed pairwise disjoint is not
obvious but follows from \cref{rem:stabilization} together with \cref{thm:murphy}).
Pick a collection of pairwise disjoint arcs $\gamma'_i$ in $L_-^{(n)}\setminus L_-^{(n-1)}$
connecting $\del W_i$ to $\del D_i$. After pushing $W_i$ along a contact vector field
tangent to $\gamma'_i$ and $L_-$, and supported in a small neighborhood of
$\gamma'_i$, we obtain a new loose chart $W'_i$ for $L_-$ such that $W'_i\cap L_+=\emptyset$
and $D'_i:=W'_i\cap L_-$ intersects $\Int D_i$. A Darboux ball $B'$ containing all the
balls $W'_i$ in its interior can be obtained by a connect sum operation along
transverse arcs.

Pick a disk $D_i'' \subset \left[ \left( \Int D_i' \cap \Int D_i \right) \setminus \{x_i\} \right]$.
We claim that there is a Legendrian isotopy $g''_u:D'_i \to W'_i$ such that
\begin{itemize}
\item $g''_0$ is the inclusion,
\item $g''_u=g''_0$ near $\del D'_i$,
\item $g''_1=g''_0$ in $D'_i\setminus D_i''$,
\item $g''_1(D'_i)$ has a loose chart $U_i$ such that
$f([0, 1] \times U_i) \subset (\Int C_i) \setminus (L_+ \cup f([0, 1] \times (D_i\setminus D_i'')))$,
\item $f([0, 1] \times g''_1(D_i'')) \subset \Int C_i \setminus (L_+ \cup f([0, 1] \times (D_i \setminus D_i'')))$.
\end{itemize}

Using \cref{rem:stabilization}, we may first construct a formal Legendrian
isotopy supported in a small neighborhood of a point of $D_i''$
satisfying the above properties. It is a priori genuine only at $u=0$ and $u=1$.
But using looseness, \cref{thm:murphy} (with $p=1$, relative to $\del D'_i$ in the source
and to $\del W'_i$ in the target) allows us to deform it into a genuine Legendrian
isotopy with fixed end-points.
Extend this Legendrian isotopy to a contact isotopy, still denoted $g''$,
supported in the union of all $W'_i$, and thus in $B'$. Since $g''$ has support in this union, 
$g''([0, 1] \times L_-) \cap L_+$ is empty. Hence there exists a positive $\epsilon$ such that 
$f_t(g''([0, 1] \times L_-)) \cap L_+$ is empty for all $t \leqslant \epsilon$.

We set $g'_t=(g''_{t/\epsilon})^{-1}$
for $t\leq \epsilon$ and $g'_t=(g''_\epsilon)^{-1}$ for $t \geq \epsilon$ (or rather a
version smoothed at $t=\epsilon$). Our choice of $\epsilon$ ensures that 
$f_t \circ (g'_t)^{-1}(L_-) \cap L_+ = \emptyset$ for $t \leqslant \epsilon$. In addition
$U_i$ is a loose chart for $(g'_t)^{-1}(L_-)$ for all $t\geq \epsilon$.
\begin{figure}[ht]
  \centering
  \begin{subfigure}[b]{6cm}
    \includegraphics[width=6cm]{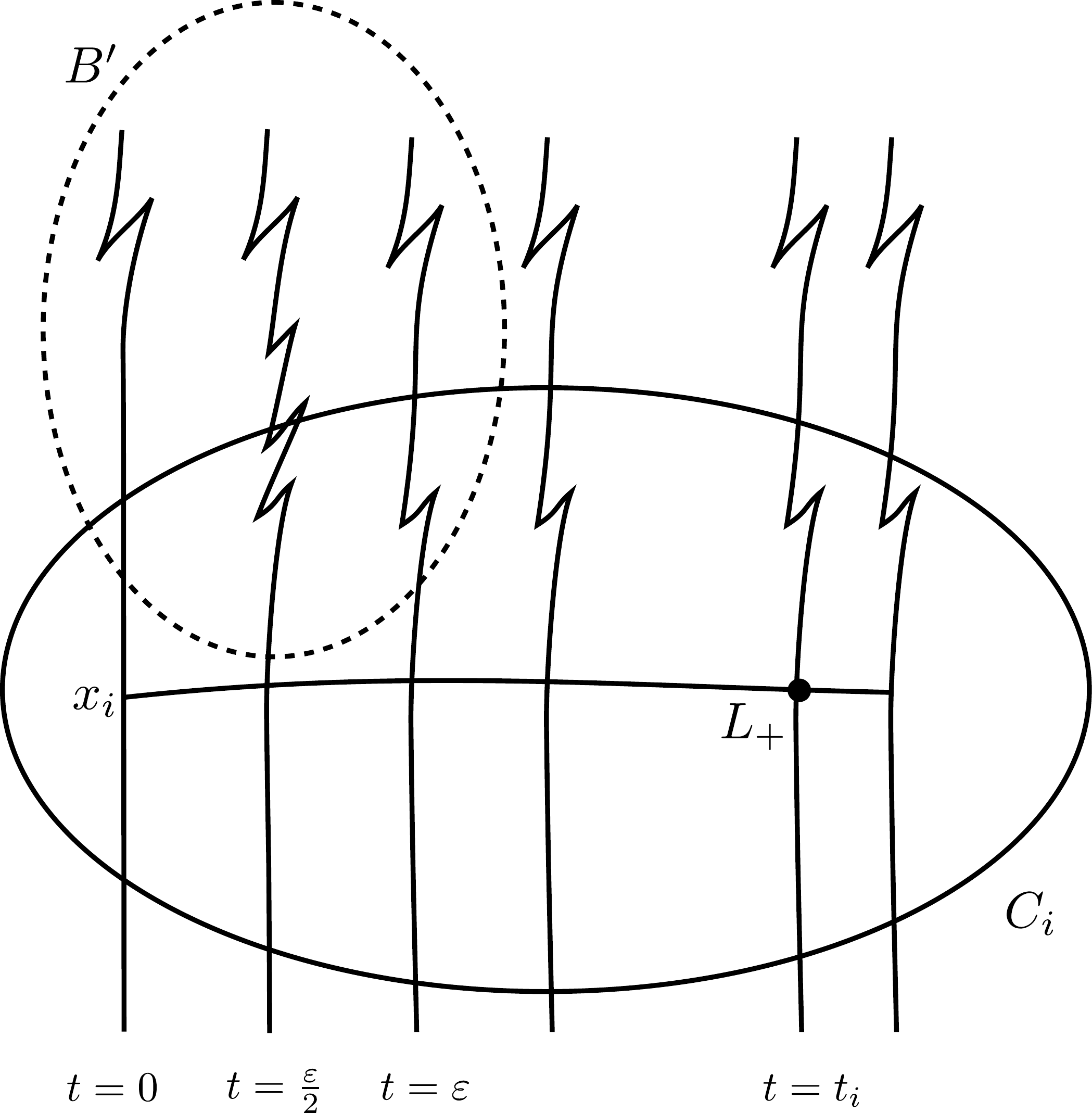}
    \caption{}
    \label{fig:step2_2_1}
  \end{subfigure}
  \begin{subfigure}[b]{6cm}
    \includegraphics[width=6cm]{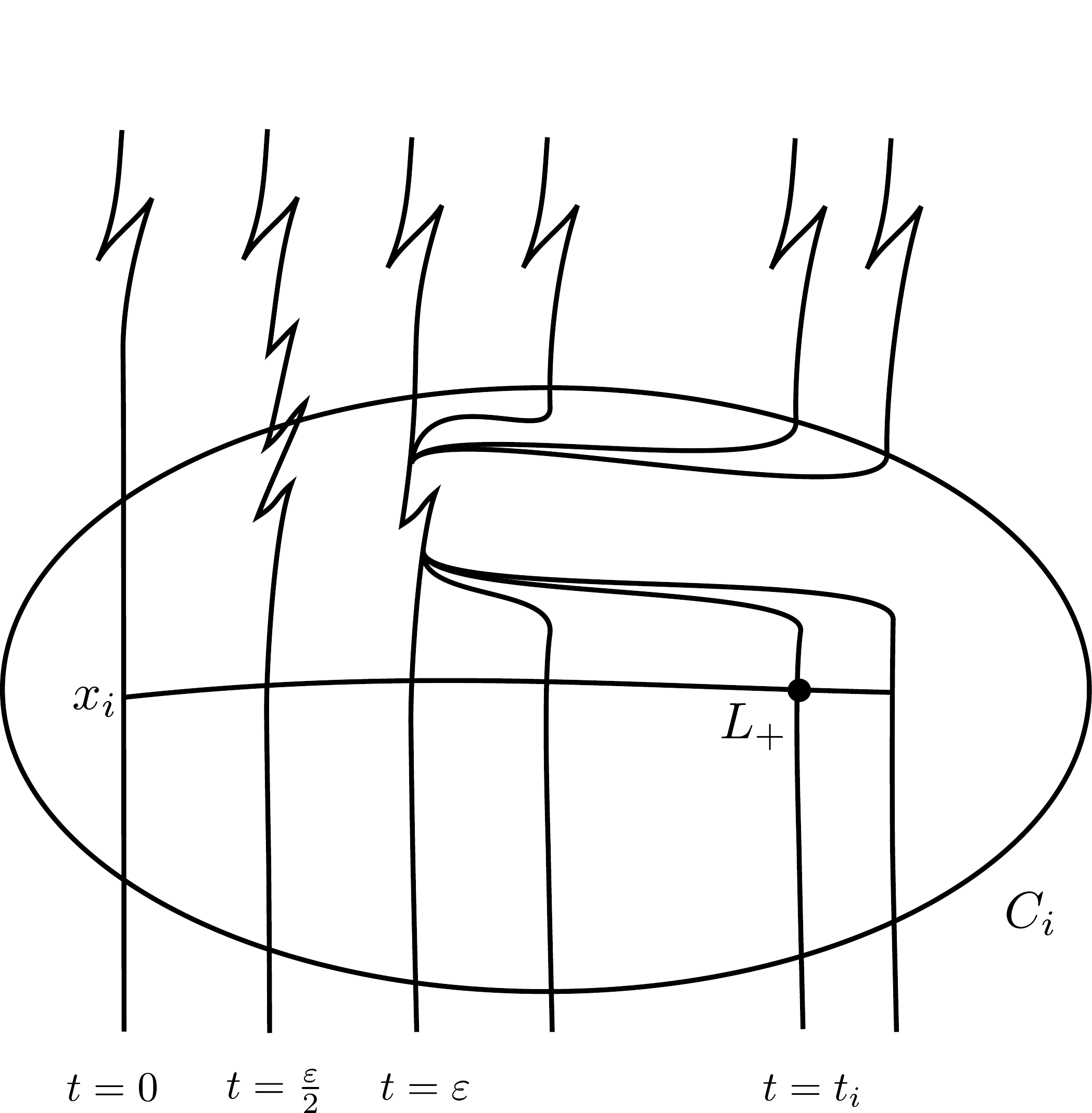}
    \caption{}
    \label{fig:step2_2_2}
  \end{subfigure}
  \caption{}
  \label{fig:step2_2}
\end{figure}

But $f_t(U_i)$ is moving, so the last thing to do is to fix this loose chart
using an isotopy $h$.
For this, we cut off the contact vector field generating the isotopy
$f_t$ by multiplying the corresponding section of $T V/\xi$
by a function $\rho \co [0, 1] \times V \to \mathbb{R}$ equal to $1$ on the image of
$[0,1] \times U_i$ by the embedding $(t, x) \mapsto (t, f_t(x))$, and supported in
$[0, 1] \times (\Int C_i \setminus (\gamma_i \cup L_+))$. We obtain a contact isotopy $h_t$
such that $h_t=f_t$ on $U_i$ for all $t$ and hence the family of isotropic embeddings
$h_t^{-1} \circ f_t\circ (g'_t)^{-1}(L_-)$ has a fixed loose chart
$U_i$ for all $t \geq\epsilon$, ensuring \cref{it:fixed_chart}.
It is straightforward to check \cref{it:same_coll,it:in_D_i'',it:same_disks}
using the properties of $g''_1$ listed above. The support of $h$ is in $B$
thanks to \cref{it:tubdarboux}.

\begin{step}(See \cref{fig:step3})
There is a family $k_{t, s}:D \to C$, $(t, s) \in[\epsilon,1]\times [0, 1]$ of neat
Legendrian embeddings such that
  \begin{stepitems}
    \item\label{it:s=0} $k_{t,0}=k_t$,
    \item\label{it:t=epsilon} $k_{t, s}=k_t$ for $t$ near $\epsilon$,
    \item\label{it:bordDi} $k_{t, s} = k_t$ near $\del D$,
    \item\label{it:nocollision} $k_{t, 1}^{-1}(L_+) = \emptyset$.
  \end{stepitems}
\end{step}

\begin{figure}[ht]
  \centering
  \includegraphics[width=5cm]{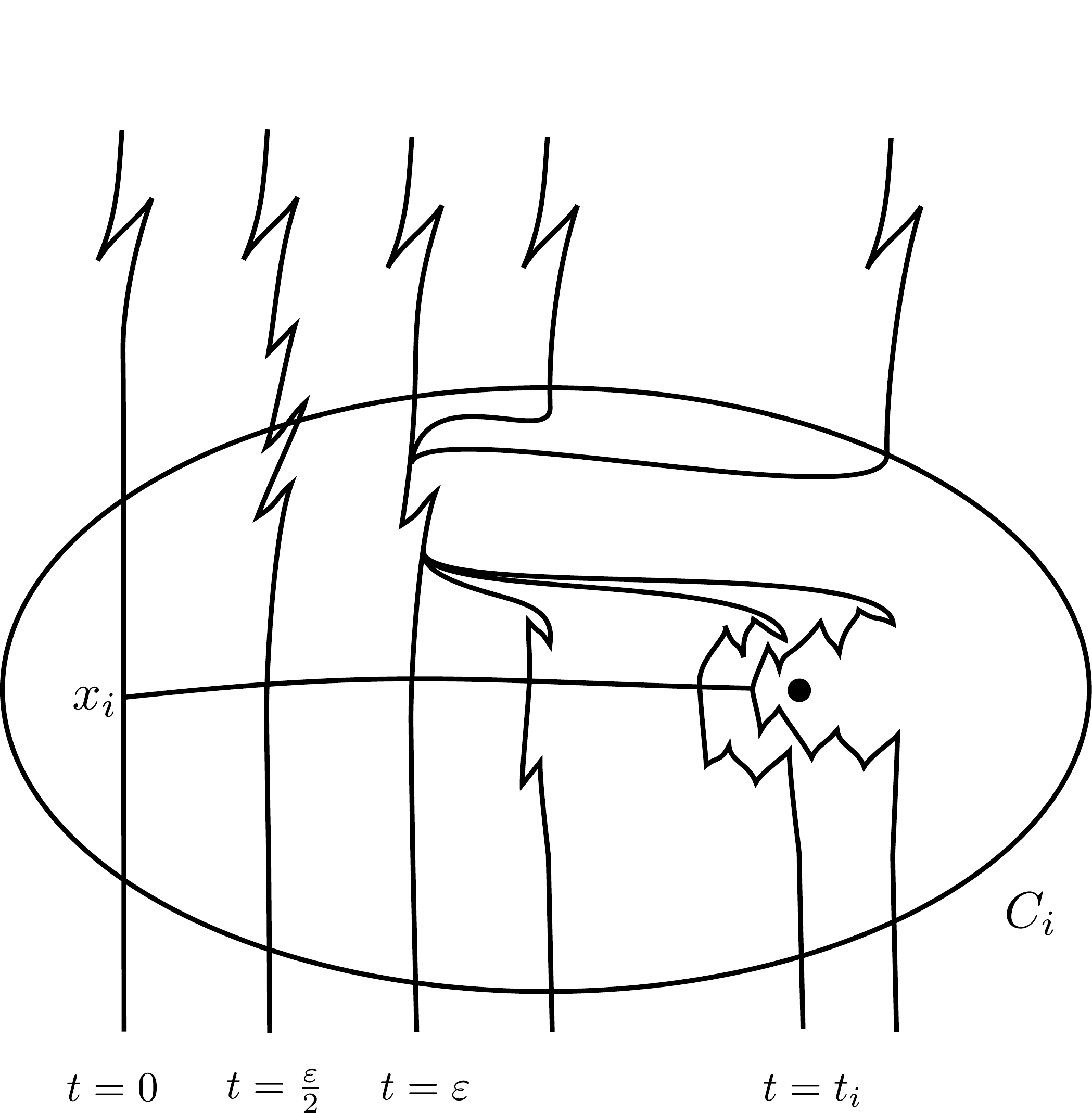}
  \caption{}
  \label{fig:step3}
\end{figure}

We fix $\eta$ in $(\epsilon, \min_i t_i)$ and
$\theta_s \co [\epsilon, 1] \times D \to [\epsilon, 1] \times D$, $s \in [0, 1/2]$, a family of embeddings,
starting with $\Id$, supported in $[\eta, 1] \times (D \setminus D'')$
and such that $\theta_{1/2}([\epsilon, 1] \times D) \cap ([t_i, 1] \times \{x_i\}) = \emptyset$.
We set $k'_{t,s}(x)=k(\theta_s(t,x))$. Thanks to \ref{it:emb}
and \ref{it:in_D_i''}, we see that for each $t\in[\epsilon,1]$ and
$s\in[0, 1/2]$, $k'_{t,s}:D\to C$ is a neat embedding. This satisfies
\cref{it:s=0,it:t=epsilon,it:bordDi,it:nocollision} but has no reason to be
Legendrian for $s>0$. However, it inherits from $k_{t,0}$ the structure of a formal
Legendrian embedding in a homotopically unique way.
We will extend the set of $(t, s)$ for which $k'_{t, s}$ is defined
to $[\epsilon, 2] \times [0, 1]$, and then deform the extended family into a family whose
restriction to $[\epsilon , 1] \times [0, 1]$ will be genuinely Legendrian 
without destroying \cref{it:s=0,it:t=epsilon,it:bordDi,it:nocollision}.
In particular it is important that, for all $s$, $k'_{t, s}\rst{\partial D}$ will be 
$k'_t\rst{\partial D}$ if $t \leqslant 1$ and $k'_1\rst{\partial D}$ if $t \geqslant 1$.
For this, we will apply \cref{thm:murphy} several times using the fact
that the embeddings $k'_{t,s}$ have a fixed loose chart $U_i$ by \ref{it:fixed_chart}.
Note that we can guarantee that a fixed loose chart remains during
each application of \cref{thm:murphy} (the idea is that a loose chart contains
two disjoint loose charts, so we may use one and let the other fixed), we will not
repeat it each time. We first apply \cref{thm:murphy} with $p = 0$, to the embedding $k'_{1,\frac{1}{2}}$
in the contact manifold $C_i \setminus L_+$ to find a family of formal Legendrian
embeddings $k'_{t,\frac{1}{2}}$ for $t\in[1,2]$ so that $k'_{2,\frac{1}{2}}$ is a
genuine Legendrian embedding. Next we apply \cref{thm:murphy} with $p = 1$ to the
family $k'_{t,1/2}$ for $t\in[\epsilon,2]$ in the contact manifold $C_i\setminus L_+$
to find a family $k'_{t,s}$ for $t\in[\epsilon,2]$ and $s\in[1/2,1]$ so that
$k'_{t,1}$ is genuine and $k'_{t,1}(D_i)\cap L_+=\emptyset$. Then we apply \cref{thm:murphy}
with $p = 1$ to the family $k'_{1,s}$ for $s\in[0,1/2]$ concatenated with
$k_{t,\frac{1}{2}}$ for $t\in[1,2]$ in the contact manifold $C_i$ to find a family
$k'_{t,s}$ for $t\in[1,2]$ and $s \in[0,\frac{1}{2}]$ that is genuine for $s = 0$
and for $t = 2$. Finally we apply \cref{thm:murphy} with $p = 2$ to the family
$k'_{t,s}$ with $t\in[\epsilon,2]$ and $s\in[0,1]$ in the contact manifold $C_i$ to find
a family $k'_{t,s,r}$ with $r\in[0,1]$ such that $k'_{t,s,1}$ is
genuine. By construction, $k_{t,s}:=k'_{t,s,1}$ for $s\in[0,1]$ and $t\in[\epsilon, 1]$
is a family of neat Legendrian embeddings satisfying
\cref{it:s=0,it:t=epsilon,it:bordDi,it:nocollision}.

\begin{step}
There exists a contact isotopy $l_t$ supported in $B$ such that
$(l_t\circ f'_t)(L_-)\cap L_+=\emptyset$ for all $t$.
\end{step}

For $t\in[\epsilon,1]$ and $s\in[0,1]$, we consider the family of vector fields
$\nu_{t,s}=\frac{\d k_{t,s}}{\d s}$ along the embedding $k_{t,s}$.
Since $k_{t,s}$ is Legendrian for all $t$ and $s$, it can be extended to a smooth
family of contact vector fields $\tilde{\nu}_{t,s}$ supported in an arbitrarily
small neighborhood of the support of $\nu_{t,s}$.
In particular, we may assume using \ref{it:bordDi} that the support of
$\tilde{\nu}_{t,s}$ is contained in the interior of $C_i$. Moreover, we may assume that
$\tilde{\nu}_{t,s}=0$ near $t=\epsilon$ thanks to \ref{it:t=epsilon}.
The family of vector fields $\tilde{\nu}_{t,s}$ then integrates uniquely into a smooth
family of contact diffeomorphisms $\psi_{t,s}$ such that $\psi_{t,0}=\id$ and
$\frac{\d \psi_{t,s}}{\d s}=\tilde{\nu}_{t,s} \circ \psi_{t,s}$.
By construction and using \ref{it:s=0}, we get $k_{t,s}=\psi_{t,s}\circ k_t$,
and $\psi_{t,s}=\id$ for $t$ near $\epsilon$.  We may thus smoothly extend the
family $\psi_{t,s}$ by defining $\psi_{t,s}=\id$ for $t\leq \epsilon$. We now set
$l_t=\psi_{t,1}$, which is supported in $C$ hence in $B$, and get
$(l_t\circ f'_t)(L_-)\cap L_+=\emptyset$ for all $t$.

\paragraph{Conclusion}

Define $g_t := h_t \circ (l_t)^{-1}$ which is a contact isotopy supported in $B$.
We have achieved $g_t^{-1} \circ f_t \circ (g'_t)^{-1}(L_-) \cap L_+ = \emptyset$
for all $t$. The claim now follows from \cref{lem:disjoint}.

\section{Commutators, fragments and conjugates}
\label{sec:commutators_and_fragments}

In this section we prove all theorems stated in the introduction, except
\cref{thm:decomp_intro} which was already subsumed in \cref{thm:decomposition}. We
first need a lemma which is used in several of them.

\begin{lemma}
\label{lem:ball_conj}
  Let $(V, \xi)$ be a contact manifold and let $G$ denote either 
  $\Do(V, \xi)$ or its universal cover. Let $B$ be a Darboux ball inside
  $V$, $p$ a point in the interior of $B$, and $g$ an element of $G$. If
  $g(p)$ is in $B \setminus \{p\}$ then every element of $G$ with support in the
  interior of $B$ is a product of eight conjugates of $g^{\pm1}$.
\end{lemma} 

\begin{proof}
  The interior of $B$ is isomorphic to $\mathbb{R}^{2n+1}$ equipped 
    with $\ker(dz + \lambda)$ where $\lambda$ is the radial Liouville form on $\mathbb{R}^{2n}$,
  and $p$ gets mapped to the origin.
  Let $f$ be an element of $G$ with compact support in the interior of
  $B$. We consider the Heisenberg dilatation flow 
  $\delta_t \co (x, y, z)) \mapsto (e^{-t}x, e^{-t}y, e^{-2t}z)$. Note that each
  ball $B_r := \{ \|x\|^4 + \|y\|^4 + \|z\|^2 < r \}$ is preserved by $\delta_t$, $t \geqslant 0$.
  We fix $R$ large enough to make sure that both $g(p)$ and the support
  of $f$ are in $B_R$, and that $p$ is in $g(B_R)$.
  We cut $\delta_t$ between radius $4R$ and $5R$ to ensure that it extends to
  a global flow $\varphi$ of $G$ which compresses $V' := B_{4R}$ onto $\{p\}$.
  We also cut it between radius $2R$ and $3R$ to get a flow $\theta$ with
  support in $V'$ and compressing $V'' :=  B_{2R}$ onto $\{p\}$. It
  only remain to apply Rybicki's theorem and \cref{prop:huit_conj}.
\end{proof}

\begin{proof}[Proof of \cref{thm:trente_deux_conj,thm:c0}]
  \cite{Giroux_ICM} ensures that $(V, \xi)$ has a supporting open book
  with Weinstein pages.  Then \cref{prop:leg_skeleta} gives isotropic
  complexes $L_-$ and $L_+$, and a contact flow $\varphi_t$ which retracts the
  complement of one complex onto the other.

  Let $G$ be either $\Do(V, \xi)$ or its universal cover.  According to
  \cref{lem:disjoint}, there is a $C^0$-neighborhood of the identity in
  $G$ or $\tildeG$ such that all elements can be written as a product
  $f^- \circ f^+$ where $f^\pm$ has compact support in the complement of $L_\pm$.
  Alternatively, if we assume the existence of an open book with
  flexible pages, then \cref{thm:decomposition} decomposes any element
  of $G$ or $\tildeG$ as $g \circ f^- \circ f^+ \circ g'$ where $g$ and $g'$ have support in
  Darboux balls $B$ and $B',$ and each $f^\pm$ is relative to $\varphi_1(L_\pm)$
  for some contact isotopy $\varphi_t$. Up to conjugating $f^{\pm}$
  by this contact isotopy $\varphi$, we may assume that $f^{\pm}$ is relative
  to $L_\pm$.

  According to \cref{prop:displacing} (see also \cref{rem:conicsing}),
  if $\psi_t$ is a positive or negative
  contact isotopy then $\psi_\epsilon$ displaces $L_- \cup L_+$ for all $\epsilon$ in some interval
  $(0, \epsilon_0]$. In the flexible page case, after reducing $\epsilon_0$ if needed, we can
  also assume there exists $p$ in $B$ (resp. $p'$ in $B'$) such that $\psi_\epsilon(p)$
  is in $B \setminus \{p\}$ (resp.  $\psi_\epsilon(p')$ is in $B' \setminus \{p'\}$). Such a small $\epsilon$
  is now fixed until the end of the proof.

  In both cases we have a decomposition into $k+2$ contact
  transformations, where $k = 0$ in the first case and $k = 2$ in the
  second case. We only need to prove that each piece is a product of
  eight conjugates of $\psi_\epsilon^{\pm 1}$. This property is invariant under
  conjugation so we are free to conjugate each piece (separately).  We
  apply \cref{prop:huit_conj} to each piece. Pieces with support in
  Darboux balls are handled by \cref{lem:ball_conj}.

  We now explain how to deal with $f^-$, the case of $f^+$ being
  completely symmetric.  Let $V'=V\setminus L_-$ and $V''=V\setminus N$
  where $N$ is a compact neighborhood of $L_-$ so small
  that $\supp(f^-) \subset V''$, $\psi_\epsilon(L_+) \subset V''$ and $L_+ \subset \psi_\epsilon(V'')$.
  We then cut the Hamiltonian defining $X_t$ near $L_-$ in order to obtain
  a flow $\theta_t$ which has compact support in $V'$ but agrees
  with $\varphi_t$ in $V''$. In particular $\theta$ compresses $V''$ onto $L_+$.  It only
  remain to apply Rybicki's theorem and \cref{prop:huit_conj}.
\end{proof}

\begin{proof}[Proof of \cref{thm:main}]
  Let $G$ be either $\Do(V, \xi)$ or its universal cover. Let $L_-$ and $L_+$ be
  isotropic complexes associated to an open book with flexible pages as in the
  proof of \cref{thm:trente_deux_conj}. We claim there is some $\psi_\epsilon$ in $G$
  which displaces $L := L_- \cup L_+$ and is a product of $n = (\dim V - 1)/2$
  elements with support in Darboux balls. Indeed, let $\phi_t$ be a small Reeb flow
  displacing $L$ for all positive $t$ (see \cref{prop:displacing}). \Cref{lem:frag}
  (see also \cref{rem:cellcomplex}) gives contact isotopies
  $\phi_0$, \dots, $\phi_n$ such that 
  $\psi_t := \phi_{0, t} \circ \cdots \circ \phi_{n, t}$ coincides with $\phi_t$
  near $L$ for small $t$. Hence there is some $\epsilon$ such that
  $\psi_\epsilon$ displaces $L$. 

  The proof of \cref{thm:trente_deux_conj} proves more generally that every
  element $f$ of $G$ is a product of $32$ conjugates of $\psi_\epsilon$, hence
  of $32(n+1)$ elements $f_i$ with support in Darboux balls $B_i$. Let $g$ be
  any element of $G$ displacing some point $p$ (this means any non-trivial
  element in the case of $\Do(V, \xi)$, or any element not lying over the
  identity in the universal cover case). Let $B$ be a Darboux ball containing
  both $p$ and $g(p)$ in its interior. Because $G$ acts transitively on the set
  of Darboux balls (see e.g.  \cite[Theorem~2.6.7]{Geiges_book}) each $f_i$ is
  conjugated to an element with support in $B$, hence to a product of eight
  conjugates of $g^{\pm1}$ by \cref{lem:ball_conj}. So $f$ itself is conjugated
  to a product of at most $256(n+1)$ conjugates of $g^{\pm1}$.
\end{proof}

\printbibliography
\end{document}